\documentclass[reqno]{amsart}

\usepackage{amssymb,amsmath,amsthm,amsfonts,xcolor} 

\usepackage{comment,graphicx}

\usepackage{enumitem}

\usepackage[colorlinks=true, pdfstartview=FitV, linkcolor=blue, citecolor=red]{hyperref}

\usepackage{color,setspace,wrapfig,multicol}

\usepackage[labelformat=empty]{subcaption}

\usepackage[mathscr]{eucal}

\usepackage[left=3cm,right=3cm,top=2.5cm,bottom=2.5cm]{geometry}

\allowdisplaybreaks
\numberwithin{equation}{section}

\newtheorem{theorem}{Theorem}[section]
\newtheorem{proposition}[theorem]{Proposition}
\newtheorem{lemma}[theorem]{Lemma}
\newtheorem{corollary}[theorem]{Corollary}

\theoremstyle{definition}
\newtheorem{definition}[theorem]{Definition}

\newtheorem{remark}[theorem]{Remark}

\newcommand{\be}{\mathbf{e}}
\newcommand{\bn}{\mathbf{n}}

\newcommand{\cA}{\mathcal{A}}
\newcommand{\cB}{\mathcal{B}}
\newcommand{\cC}{\mathcal{C}}
\newcommand{\cI}{\mathcal{I}}
\newcommand{\cP}{\mathcal{P}}
\newcommand{\cH}{\mathcal{H}}
\newcommand{\cR}{\mathcal{R}}
\newcommand{\cS}{\mathcal{S}}
\newcommand{\cT}{\mathcal{T}}

\newcommand{\N}{\mathbb{N}}
\newcommand{\R}{\mathbb{R}}
\renewcommand{\S}{\mathbb{S}}

\newcommand{\C}{\mathbb{C}}

\newcommand{\sA}{\mathscr{A}}
\newcommand{\sB}{\mathscr{B}}
\newcommand{\sC}{\mathscr{C}}
\newcommand{\sD}{\mathscr{D}}
\newcommand{\sF}{\mathscr{F}}
\newcommand{\sS}{\mathscr{S}}
\newcommand{\sU}{\mathscr{U}}

\newcommand{\Trace}{\mathrm{Tr}}
\newcommand{\Id}{\mathrm{I}}

\newcommand{\dist}{\mathrm{dist}}

\newcommand{\intpart}[1]{\lfloor #1\rfloor}
\renewcommand{\tilde}{\widetilde}
\renewcommand{\d}{\mathrm{d}}
\newcommand{\sd}{\mathrm{sd}}

\newcommand{\sign}{\mathrm{sign}\,}
\newcommand{\loc}{ {\mathrm{loc}} }

\newcommand{\p}{\partial}
\newcommand{\pOmega}{\partial^{\Omega}}

\newcommand{\Int}[1]{\mathrm{Int}(#1)}
\newcommand{\cl}[1]{\overline{#1}}

\newcommand{\cn}{\mathtt{n}}

\setlist[itemize]{leftmargin=6mm} % Remove the indentation
\begin{document}

%\onehalfspace

\title[Evolution of droplets]{Minimizing movements for forced anisotropic curvature flow of droplets}

\author[Sh. Kholmatov] {Shokhrukh Yu. Kholmatov} 
\address[Sh. Kholmatov]{University of Vienna,  Oskar-Morgenstern Platz 1, 1090 Vienna 
(Austria)}
\email{shokhrukh.kholmatov@univie.ac.at}

\keywords{anisotropy, capillarity functional, droplet, anisotropic curvature flow, minimizing movements, consistency}

\subjclass[2010]{53C44, 49Q20, 35A15,  35D30, 35D35}

\date{\today}

\begin{abstract}
We study forced anisotropic curvature flow of droplets on an inhomogeneous horizontal hyperplane. As in \cite{BKh:2018} we  establish the existence of smooth flow, starting from a regular droplet and satisfying the prescribed anisotropic Young's law, and also the existence of a $1/2$-H\"older continuous in time minimizing movement solution starting from a set of finite perimeter. Furthermore, we investigate various properties of minimizing movements, including comparison principles, uniform boundedness and the consistency with the smooth flow.
\end{abstract}

\maketitle

\section{Introduction}

In this paper we are interested in the  forced anisotropic mean curvature flow of droplets on an inhomogeneous hyperplane. Representing the droplets by subsets of the halfplane 
$$
\Omega:=\R^{n-1}\times (0,+\infty)
$$ 
and the relative adhesion coefficient of the hyperplane $\p\Omega:=\R^{n-1}\times\{0\}$ by a function $\beta:\p\Omega\to\R,$ we write the corresponding evolution equation of droplets $\{E(t)\}_{t\in[0,T)}$ as
\begin{equation}\label{aniso_mce}
\begin{cases}
v_{\Gamma(t)}(x) = -\kappa_{\Gamma(t)}^\Phi (x) - f(t,x) & \text{for $t\in(0,T)$ and $x\in \Gamma(t),$}\\[2mm]
\p\Gamma(t) \subset \p\Omega & \text{for $t\in [0,T) ,$}\\[2mm]
\nabla \Phi(\nu_{\Gamma(t)}(x)) \cdot \be_n = -\beta(x) & \text{for $t\in (0,T)$ and $x\in \p\Gamma(t),$}\\[2mm]
E(0)=E_0,
\end{cases}
\end{equation}
where $\Gamma(t):=\Omega \cap\p E(t)$ is the free boundary of $E(t)$ in $\Omega,$ $\Phi$ is an \emph{even anisotropy} in $\R^n,$ i.e., a positively one-homogeneous even convex function in $\R^n$ satisfying 
\begin{equation}\label{norm_bounds}
c_\Phi|x| \le \Phi(x) \le C_\Phi|x|,\quad x\in\R^n, 
\end{equation}
for some $0<c_\Phi\le C_\Phi,$ 
$v_{\Gamma(t)}$ and $\kappa_{\Gamma(t)}^\Phi$ are the normal velocity and the anisotropic mean curvature of $\Gamma(t),$ respectively, $\nu_{\Gamma(t)}$ is the unit normal, outer to $E(t),$ $f:\R_0^+\times \Omega\to\R$  is a forcing term, $\be_n=(0,\ldots,0,1)$ and $E_0$ is the initial droplet, here $\R_0^+:=[0,+\infty)$. In the literature the third equation in \eqref{aniso_mce} is called the anisotropic Young's law or anisotropic contact-angle condition \cite{DPhM:2015}. 
We refer to solutions of \eqref{aniso_mce} as the $\Phi$-curvature flow starting from $E_0,$ with forcing $f$ and anisotropic contact angle $\beta.$ 

The following result shows that the equation \eqref{aniso_mce} is well-posed.

\begin{theorem}\label{teo:short_intro}
Let $\Phi$ be an elliptic $C^{3+\alpha}$-anisotropy in $\R^n,$ $ f\in C^{\frac{\alpha}{2},\alpha}(\R_0^+\times \cl{\Omega}),$ $\beta\in C^{1+\alpha}(\p\Omega)$ with $\|\beta\|_\infty<\Phi(\be_n)$ and  $E_0\subset\Omega$ be a bounded set such that $\Gamma_0:=\Omega \cap\p E_0$ is a $C^{2+\alpha}$-hypersurface with boundary satisfying 
$$
\p\Gamma_0\subset \p\Omega\quad\text{and}\quad \nabla\Phi(\nu_{\Gamma_0})\cdot\be_n = \beta\quad\text{on $\p\Gamma_0$},
$$
where $\alpha\in (0,1].$ Then there exist a maximal time $T^\dag>0$ and a unique $\Phi$-curvature flow $\{E(t)\}_{t\in[0,T^\dag)}$ starting from $E_0,$ with forcing $f$ and anisotropic contact angle $\beta.$
\end{theorem}

In section \ref{sec:smooth_flow}, we establish this theorem in a more general form (see Theorem \ref{teo:max_time_exist}) by following the arguments presented in \cite{BKh:2018, KKR:1995}. Specifically, we begin by introducing a convenient parametrization and linearize the problem near the initial and boundary conditions. Subsequently, we employ the Solonnikov method \cite{Solonnikov:1965} to solve the linearized problem, and next, utilizing fixed-point arguments in the H\"older spaces, we solve the nonlinear problem for small time intervals. Finally, through iterative application of this short-time argument, we extend the solution till the maximal time.
Since the fixed-point method is quite robust, the flow is stable with respect to small perturbations of initial condition, $\Phi,$ $\beta$ and $f$ (Theorem  \ref{teo:stability_mcf}). 
Moreover, as in the Euclidean case (see e.g. \cite{BKh:2018,Kholmatov:2023}), the smooth $\Phi$-curvature flow satisfies a strong comparison principle (Theorem \ref{teo:comparison_mcf}), which in particular, shows the uniqueness of the flow.
The stability of the smooth flow  allows to flow tubular neighborhoods of initial sets (Theorem \ref{teo:flow_tubular_nbhd}); we anticipate here that the evolution of tubular neighborhoods is an important  ingredient in the proof of the consistency of GMM (Definition \ref{def:gmm}) with the smooth flow.

The evolution equation \eqref{aniso_mce} can be seen as mean curvature flow of hypersurfaces with a prescribed Neumann-type boundary condition. There are quite a few results related to the well-posedness of the  classical mean curvature flow with Neumann boundary condition,  see e.g. \cite{AW:1994,BKh:2018, BGM:2023, GS:1991, Guan:1996,HL:2021, KKR:1995}; see also \cite{GOT:2021, Huisken:1989,OU:1991III,White:2021} for mean curvature flow with Dirichlet boundary conditions.

When $f\equiv0,$ the evolution equation \eqref{aniso_mce} is a gradient flow for the functional  
\begin{equation}\label{def:capil_funcos}
\sC_{\beta}(E):= P_\Phi(E,\Omega) + \int_{\p\Omega} \beta \chi_E\,d\cH^{n-1},\quad E\in BV(\Omega;\{0,1\}),
\end{equation}
where for simplicity we drop the dependence of $\sC_\beta$ on $\Phi,$
$$
P_\Phi(E,U):=\int_{U\cap \p^*E} \Phi(\nu_E)d\cH^{n-1}
$$
is the $\Phi$-perimeter of $E$ in an open set $U,$ and $\p^*E$ and $\nu_E$ are the reduced boundary and the generalized outer unit normal of $E.$ To maintain the $L^1(\Omega)$-lower semicontinuity and coercivity of the capillary functional we always assume 
\begin{equation}\label{beta_condio}
\exists \eta \in (0,1/2):\quad \|\beta\|_\infty\le (1-2\eta)\Phi(\be_n).
\end{equation}
Under this assumption and a priori estimates (see  \eqref{qanday_kungil_uzsam} below)
\begin{equation}\label{eq:coersive_capillar}
\eta P_\Phi(E) \le \sC_\beta(E) \le P_\Phi(E),\quad E\in \cS.
\end{equation}

In the literature, $\sC_{\beta}$ is usually referred as the anisotropic capillary functional. Originated to the work of Young, Laplace, Gauss and others, this functional allows to consider more general classes of anisotropies $\Phi$ (such as crystalline) and relative adhesion coefficients $\beta$ not necessarily constant (see e.g. \cite{DPhM:2015,Finn:1986,Maggi:2012}). The global minimizers of this functional (usually under a volume constraint) are related to the equlibrium shapes of liquid or crystalline droplets in the container, which sometimes are called Winterbottom shapes \cite{Kholmatov:2024,KSch:2024,Maggi:2012}. Therefore, the problems, such as the existence of minimizers, the regularity of their free boundaries and contact sets, the validity of an anisotropic version of Young's contact-angle law, and the characterization of the shape of the minimizers, have been extensively investigated and addressed in numerous papers in the literature (see e.g. \cite{BB:2012,CM:2007,DGBWQ:2004,DPhM:2015,Finn:1986,JWXZ:2023,Kholmatov:2024,KSch:2024,Maggi:2012} and the references therein). 

To study a weak evolution of droplets, let 
$$
\cS:=\{E\in BV(\Omega;\{0,1\}):\,\, E = E^{(1)}\}
$$
be the metric space endowed with the $L^1(\Omega)$-distance 
$
d(E,F):=|E\Delta F|,
$
where $E^{(1)}$ is the set of points of density $1$ for $E,$ i.e. 
$$
E^{(1)}:=\{x\in\R^n:\,\,\lim\limits_{r\to0^+}\,r^{-n}|B_r(x) \setminus E|=0\},
$$
and let 
\begin{equation*}
\sF_{\beta,f}(E;E_0,\tau,k) :=
\begin{cases}
|E\Delta E_0| & k=0,\\[3mm]
\displaystyle
\sC_{\beta}(E) + \frac{1}{\tau}\int_{E\Delta E_0} \d_{E_0}dx +  \int_k^{k+1}ds \int_{E} f(\tau s, x)\,dx  & k\ge1
\end{cases}
\end{equation*}
be the anisotropic capillary Almgren-Taylor-Wang functional with a nonautonomous (time-dependent) forcing, which generalizes the isotropic setting of \cite{BKh:2018}, where $E,E_0\in \cS,$ $\tau>0,$  $k\in\N_0:=\N\cup\{0\},$ $\d_E(x):=\dist(x,\p E)$ and $f$ is a suitable forcing term.  
When $f\equiv0,$ we shortly write $\sF_{\beta}.$

\begin{definition}[\textbf{Generalized minimizing movements \cite{DeGorgi:93}}]\label{def:gmm}
$\,$
\begin{itemize}
\item[\rm (a)]  Given $\tau>0,$ a family $\{E(\tau,k)\}_{k\in\N_0}\subset\cS$ is called a (discrete) \emph{flat flow} starting from $E_0$ provided that $E(\tau,0):=E_0,$
$$
E(\tau,k) \in {\rm argmin}\,\,\,\sF_{\beta,f}(\cdot; E(\tau,k-1),\tau,k),\quad k\ge1.
$$

\item[\rm(b)] A family $\{E(t)\}_{t\in\R_0^+}$ is called a \emph{generalized minimizing movement} (shortly, GMM) starting from $E_0$ if there exist a sequence $\tau_i\to0^+$ and flat flows $\{E(\tau_i,\cdot)\}$ such that 
\begin{equation}\label{L1_convergence_gmm}
\lim\limits_{i\to+\infty}\,|E(\tau_i,\intpart{t/\tau_i}) \Delta E(t)| = 0,\quad t\ge0,
\end{equation}
where $\intpart{x}$ is the integer part of $x\in\R.$
\end{itemize}

\noindent
The collection of all GMMs starting from $E_0$ and associated to $\sF_{ \beta,f}$ will be denoted by $GMM(\sF_{\beta,f},E_0).$
\end{definition}

In applications it is enough to establish \eqref{L1_convergence_gmm} in any finite interval $[0,T)$ (thus, different $T$ may require different sequences $\tau_j\to0^+,$ but at the end, we can use a diagonal argument).

Starting from the seminal papers \cite{ATW:1993,DeGorgi:93,LS:1995}, the minimizing movement approach has been employed in numerous papers especially in proving the existence of weak (anisotropic) mean curvature flows (see e.g. \cite{ChMP:2015,ChMNP:2019,ChMP:2017}). Moreover, the robustness of the method allows for applications in other settings such as in (anisotropic) mean curvature evolution in Finsler geometry with forcing  \cite{ChDgM:2023}, a volume preserving mean curvature flow \cite{MSS:2016}, mean curvature flow with Dirichlet and Neumann-type boundary conditions \cite{BKh:2018,MT:1999} (see also Theorem \ref{teo:intro_existGMM}), a mean curvature evolution of bounded Caccioppoli partitions including anisotropies and forcing \cite{BChKh:2021,BKh:2018siam}. 

The first main result of this paper is the following 

\begin{theorem}[\textbf{Existence of generalized minimizing movements}]\label{teo:intro_existGMM}
Assume that
\begin{itemize}
\item[\rm(a)]  
\begin{equation}\label{hyp:1}
f\in L_\loc^1(\R_0\times \R^n)\quad\text{and}\quad f^-\in L_\loc^1(\R_0^+;L^1(\R^n)), \tag{H1}
\end{equation}

\item[\rm(b)]
\begin{equation}\label{hyp:2}
\forall T>0\,\,\,\exists \gamma_T>0:\,\,\,
\sup_{0<|A|<\omega_n\gamma_T^n,\,0\le t\le T}\,\,\,\,\tfrac{1}{|A|^{\frac{n-1}{n}}}\int_A |f(t,x)|dx \le \tfrac{c_\Phi \eta n\omega_n^{1/n}}{4}, \tag{H2}
\end{equation}

\item[\rm(c)]
\begin{equation}\label{hyp:f0}
\limsup\limits_{\tau\to0^+}\frac{1}{\tau}\int_0^\tau ds \int_{\R^n}|f(s,x)|dx \in [0,+\infty),\tag{H3} 
\end{equation}

\item[\rm(d)] for any $T>0$  either 
\begin{equation}\label{hyp:f1}
c_T:=\sup_{t\in [0,T]}\,\|f(t,\cdot)\|_{L^\infty(\R^n)}<+\infty \tag{H4'}
\end{equation}
or there exists $c_T>0$ such that 
\begin{equation}\label{hyp:f2}
\int_{\R^n} |f(s,x)-f(s+\tau,x)| dx \le c_T\tau,\quad s,s+\tau\in [0,T], \,\,\,\tau>0. \tag{H4''}
\end{equation}
\end{itemize}
Then for any $E_0\in \cS,$ $GMM(\sF_{\beta,f},E_0)$ is nonempty. Moreover, there exists $C_0:=C_0(\Phi,\beta,f,E_0)>0$ such that for any $E(\cdot)\in GMM(\sF_{\beta,f},E_0)$  
\begin{equation}\label{gmm_holder_cont}
|E(t)\Delta E(s)|\le C_0|t-s|^{1/2},\quad s,t>0\,\,\,\text{with $|t-s|<1$}.
\end{equation}
If $|\p E_0|=0,$ then \eqref{gmm_holder_cont} holds for all $s,t\ge0.$ 

Furthermore, assume that $E_0$ is bounded and for $T>0,$
\begin{equation}\label{hyp:fb2}
\exists R_T,a_T,b_T>0:\quad f^-(t,x) \le a_T+b_T|x|,\quad t\in [0,T],\,\,|x|\ge R_T. \tag{H5}
\end{equation}
Then each $E(\cdot)\in GMM(\sF_{\beta,f},E_0)$ is bounded in $[0,T]$, i.e., there exists $\bar R>0$ such that $E(t)\subset B_{\bar R}(0)$ for any $t\in[0,T].$
\end{theorem}

Some comments on assumptions (a)-(d) are in order.

\begin{itemize}
\item The hypothesis (a) is necessary for the well-definiteness of $\sF_{\beta,f}$ and is related to the prescribed curvature functional.

\item The condition (b) will be used in establishing uniform density estimates in Theorem \ref{teo:density_ests} and also in proving the boundedness of minimizers.

\item The hypothesis (c) is a technical assumption implying $\sF_{\beta,f}(E_0;E_0,\tau,k)<+\infty$ for any $E_0\in\cS,$ and will be necessary to estimate the forced capillary energies of flat flows $E(\tau,k)$ with that of $E(\tau,0)$ (see e.g.  \eqref{gako_ubil_on}, \eqref{furo0o} and subsequent estimates). 

\item An example of forcing $f$ satisfying (a)-(c) is  $f(t,x)=a(t)h(x)$ for some $a\in L^\infty(\R_0^+)$ and $h\in L^1(\R^n)\cap L^{p}(\R^n)$ for some $p\ge n.$

\item Assumptions \eqref{hyp:f1}  and \eqref{hyp:f2} in (d) are two (in general different) sufficient conditions for the existence and local $1/2$-time H\"older continuity of GMMs.

\item In \cite{ChDgM:2023} the authors established the local uniform boundedness of GMM for bounded forcing terms using comparison with balls. In this paper we show the same property holds also for forcing terms with at most linear growth, using comparison with Winterbottom shapes in place of balls (see section \ref{subsec:unifbounded_gmm}). 
\end{itemize}

To prove Theorem \ref{teo:intro_existGMM} we apply the already well-established machinery of Almgren-Taylor-Wang and Luckhaus-Sturzenhecker (see e.g. \cite{ATW:1993,BKh:2018,BKh:2018siam,ChDgM:2023,LS:1995,MSS:2016}). The main difficulty here is that as in \cite{ChDgM:2023} because of the time dependence of $f,$ given a flat flow $\{E(\tau,k)\},$ the sequence 
$$
k\mapsto \sC_\beta(E(\tau,k)) + \int_k^{k+1}ds \int_{E} f(\tau s, x)\,dx
$$
is not necessarily nonincreasing. This creates  numerous technical difficulties to bound the perimeter $P(E(\tau,k))$ uniformly in $\tau$ and $k,$ which is important for the sequential compactness of $\{E(\tau,\intpart{t/\tau})\}$ in $\tau.$ To overcome such an issue we use assumption (d). It is worth to mention that under the assumption \eqref{hyp:f2} every GMM is globally $1/2$-H\"older continuous in time, i.e., in \eqref{gmm_holder_cont} the assumption $|t-s|<1$ is not necessary. This is true for instance in the case of an autonomous (time-independent) forcing $f$.

As in the Euclidean case without forcing \cite[Section 6]{BKh:2018}, minimizers of $\sF_{\beta,f}$ satisfy various comparison principles (Theorem \ref{teo:compare_minima}). They yield the following comparison principle for GMMs.

\begin{theorem}[\textbf{Comparison of GMMs}]\label{teo:comparing_gmms}
Assume that $\beta_1\ge \beta_2$ $\cH^{n-1}$-a.e. on $\p\Omega,$ $E_0^{(1)}\prec E_0^{(2)}$ and $f_1 \ge f_2$ a.e. in $\R_0^+\times \Omega.$
Then: 
\begin{itemize}
\item[\rm(a)] for any $E^{(2)}(\cdot) \in GMM(\sF_{\beta_2,f_2},E_0^{(2)})$ there exists $E^{(1)}(\cdot)_* \in GMM(\sF_{\beta_1,f_1},E_0^{(1)})$ such that 
\begin{equation}\label{minimal_gmm}
E^{(1)}(t)_* \subset E^{(2)}(t)\quad\text{for all $t\ge0;$}
\end{equation}

\item[\rm(b)] for any $E^{(1)}(\cdot) \in GMM(\sF_{\beta_1,f_1},E_0^{(1)})$ there exists $E^{(2)}(\cdot)^* \in GMM(\sF_{\beta_2,f_2},E_0^{(2)})$ such that 
$$
E^{(1)}(t)  \subset E^{(2)}(t)^* \quad\text{for all $t\ge0.$}
$$
\end{itemize}
\end{theorem}

Finally, we study the relation of GMM with the smooth flow solving \eqref{aniso_mce}. 

\begin{theorem}\label{teo:gmm_vs_smooth}
Assume that $n\le 3$ and $\Phi$ is an elliptic $C^{3+\alpha}$-anisotropy in $\R^n$ for some $\alpha\in(0,1]$, or $n \le 4$ and $\Phi$ is Euclidean. Let $\beta\in C^{1+\alpha}(\p\Omega),$ $f\in C^{\frac{\alpha}{2},\alpha}([0,+\infty)\times\cl{\Omega})$  and $\{E(t)\}_{t\in[0,T^\dag)}$  be a smooth $\Phi$-curvature flow starting from $E_0,$ with forcing $f$ and anisotropic contact angle $\beta.$ Then for any $F(\cdot)\in GMM(\sF_{\beta,f},E_0)$ 
$$
E(t) = F(t),\quad t\in [0,T^\dag).
$$
\end{theorem}

Similar consistency result in the three-dimensional Euclidean case without forcing has been recently obtained in \cite{Kholmatov:2023} using  the techniques originated to \cite{ATW:1993}. To prove Theorem \ref{teo:gmm_vs_smooth} we adapt those techniques adding anisotropy and also forcing. Note that the smallness of dimension $n$ implies that the free boundary $\pOmega E_\tau$ of minimizers $E_\tau$ of $\sF_{\beta,f}(\cdot;E_0,\tau,k)$ is a $C^{2}$-hypersurface up to the boundary \cite{DPhM:2015,DPhM:2017}, satisfying the anisotropic contact angle condition with $\beta$. This allows to establish smooth inner and outer barriers for minimizers of $\sF_{\beta,f}$ in Proposition \ref{prop:atw_nonpol}. To extend this proposition to higher dimensions one need to show that a smooth hypersurface $\Gamma\subset\Omega$ with boundary in $\p\Omega$ can be an outer or inner barrier for $\pOmega E_\tau$ either only at points of the reduced boundary or only at regular points of $\pOmega E_\tau.$ Recall that the assertion for the reduced boundary is true since there are no singular minimizing cones for $P_\Phi$ containing a halfspace (see \cite[Lemma 7.3]{ATW:1993}). However, currently not much seems known about such a behaviour of singular minimizing cones  for capillary functional $\sC_\beta.$

The paper is organized as follows. In section \ref{sec:exist_gmm} we provide a full proof of Theorem \ref{teo:intro_existGMM}. Various comparison results are established in section \ref{sec:comparison}. In section \ref{sec:smooth_flow} we establish the well-posedness of \eqref{aniso_mce}, proving Theorem \ref{teo:short_intro} in more general form, and various properties of the smooth flows. Finally, we prove Theorem \ref{teo:gmm_vs_smooth} in section \ref{sec:consistence}. We conclude the paper with an  appendix, where we obtain some a priori estimates for capillary functional and a characterization of  elliptic smooth anisotropies.

\subsection*{Acknowledgements.} 

I acknowledge support from the Austrian Science Fund  (FWF) Lise Meitner Project M2571 and Stand-Alone Project P33716. Also, I am grateful to Francesco Maggi for his discussions on the regularity of contact sets of minimizers of the capillary functional, and in particular, showing his paper \cite{DPhM:2017} with Guido De Philippis.

\section{Existence of GMM}\label{sec:exist_gmm}

\subsection*{Notation} 

Given an anisotropy $\Phi$ in $\R^n,$ the dual anisotropy is defined as 
$$
\Phi^o(x):=\max_{\Phi(y)=1}\,\,x\cdot y.
$$
The following Young inequality holds:
$$
x\cdot y \le \Phi^o(x)\Phi(y),\quad x,y\in\R^n.
$$
The set $W^\Phi:=\{\Phi^o(x)\le 1\}$ are called the \emph{Wulff shape} for $\Phi.$ With a slight abuse of the definition, the translations and scalings $W_r^\Phi(x) = x + rW^\Phi$ of $W^\Phi$ are still called Wulff shape.

We say an anisotropy $\Phi$ in $\R^n$ is $C^{k+\alpha}$ if $\Phi\in C_\loc^{k+\alpha}(\R^n\setminus\{0\}).$ 
We denote by $\nabla\Phi$ and $\nabla^2\Phi$ the spatial gradient and Hessian of $\Phi.$ 
If there exists $\gamma>0$ such that $x\mapsto \Phi(x)-\gamma|x|$ is also an anisotropy, then $\Phi$ is called \emph{elliptic}. By Proposition \ref{prop:elliptic_anis_propo} a $C^{k+\alpha}$-anisotropy with $k\ge2$ is elliptic if and only if its dual is an elliptic $C^{k+\alpha}$-anisotropy.

Given an anisotropy $\Phi$, we define $$
\d_E^\Phi(x):=\inf\{\Phi(x-y):\,\, y\in \Omega\cap \p^*E\},\qquad 
\sd_E^\Phi(x):=
\begin{cases}
\d_E^\Phi(x) & x\in \Omega\setminus E,\\
-\d_{E^c}^\Phi(x) & x\in E,
\end{cases}
\quad x\in\Omega,
$$
for $E\in \cS.$ When $\Phi$ is Euclidean, we write shortly $\d_E$ and $\sd_E.$ 

To shorten the notation we use
$$
\pOmega E:= \Omega\cap \p E
$$
and
$$
E\prec F\qquad \Longleftrightarrow \qquad E\subset F\quad \text{and}\quad \dist(\pOmega E, \pOmega F)>0
$$
for $E,F\in \cS.$ Note that  
\begin{equation}\label{compare_trunc_sdist}
E\subset  F\quad\Longleftrightarrow\quad \sd_E^\Phi \ge \sd_F^\Phi \,\,\,\text{in $\Omega$}
\qquad \text{resp.}\qquad 
E\prec F\quad\Longleftrightarrow\quad \sd_E^\Phi > \sd_F^\Phi \,\,\,\text{in $\Omega.$}
\end{equation}

The following proposition shows the connection between the regular surfaces and distance functions (see also \cite[Proposition 2.1]{Kholmatov:2023}). 

\begin{proposition}\label{prop:regular_distance}
Let $\Gamma$ be a $C^{2+\alpha}$-hypersurface (not necessarily connected and with or without boundary) in an open set $Q\subset\R^n$ for some $\alpha\in[0,1]$. Then:

\begin{itemize}
\item[\rm (a)] for any $x\in \Gamma$ there exists $r_x>0$ such that 
\begin{itemize}
\item[$\bullet$] $\Gamma$ divides $B_{r_x}(x)$ into two connected components,

\item[$\bullet$] $\dist(\cdot,\Gamma)\in C^{2+\alpha}(B_{r_x}(x)\setminus \Gamma);$
\end{itemize}

\item[\rm(b)] if $\Gamma$ is compact and has no boundary, then $\inf_{x\in \Gamma} r_x >0,$ i.e., the radius $r_x$ in (a) can be taken uniform in $x;$

\item[\rm(c)] if $\Gamma = Q\cap \p E$ for some $E\subset Q$ and $\Phi$ is an elliptic $C^{3+\alpha}$-anisotropy,  then for any $x\in \Gamma$ there exists $r_x>0$ such that $B_r(x)\subset\Omega$ and $\sd_E^\Phi \in C^{2+\alpha}(B_{r_x}(x)).$ In this case, $\nabla \sd_E^\Phi(x) = \nu_E(x)^{\Phi^o}$ for any $x\in \Gamma,$ where $\nu_E$ is the outer unit normal of $E$ and
\begin{equation}\label{eta_Phio}
\eta^{\Phi^o} = \tfrac{\eta}{\Phi^o(\eta)},\qquad 0\ne \eta\in\R^n.
\end{equation}

\item[\rm(d)] Assume that in (c), additionally, $Q$ has  a $C^{2+\alpha}$-boundary. Then under assumptions of (b), for any $x_0\in \p Q\cap \p\Gamma$ and for any $\eta\in \S^{n-1}$ with $\eta\cdot \nu_Q(x_0) <0,$ where $\nu_Q(x_0)$ is the outer unit normal of $\p Q$ at $x_0,$ one has 
$$
\sd_E^\Phi(x_0 + s\eta) - \sd_E^\Phi(x_0) = s \Big(\nu_E(x_0)^{\Phi^o}\cdot \eta + o(1)\Big)\quad\text{as $s\to0^+$}.
$$
\end{itemize}
\end{proposition}

These assertions can be proven using the local geometry of $\Gamma$ and $Q,$ i.e. passing to the local coordinates (see also \cite{DZ:2011}).

Given an elliptic $C^2$-anisotropy $\Phi$ and a $C^2$-hypersurface $\Gamma\subset\R^n$ oriented by a unit normal $\nu_\Gamma,$ the \emph{$\Phi$-curvature} of $\Gamma$ at $x\in\Gamma$ 
is defined as \cite[Section 2.2]{ATW:1993}
$$
\kappa_\Gamma^\Phi(x) := \Trace [\nabla^2\Phi^o(\nu_\Gamma(x))\nabla^2R(x)],
$$
where $R$ is any $C^2$-function in a ball $B_r(x)$ with small radius $r>0$  satisfying
$$
B_r(x) \cap \Gamma = \{R=0\} \quad\text{and}\quad \nabla R(x) = \nu_E(x).
$$
Writing $\Gamma$ as a graph near $x,$
one can show that $\kappa_\Gamma^\Phi(x)$ is independent of the choice of $R.$ 
When $\Gamma=\pOmega E$ for some $E\subset \Omega,$ we orient $\Gamma$ along the outer unit normal of $E$ and write 
$$
\kappa_E^\Phi:=\kappa_\Gamma^\Phi.
$$
By this convention, convex sets admit a nonnegative $\Phi$-curvature. We also set 
$$
\|II_E\|_\infty:= \sup_{x\in \Gamma} \,|II_\Gamma(x)|,
$$
where $II_\Gamma$ is the second fundamental form of $\Gamma.$ 

Recall that if $E_t = (\Id + t X)(E)$ is a $C^1$-perturbation of $E\subset\Omega$ for some $X\in C_c^1(U;\R^n)$ with $X\cdot \be_n=0$ on $\p\Omega,$  then the first variation of the capillary functional $\sC_\beta$ at $E$ is computed as \cite{JWXZ:2023}
\begin{align}\label{euler_lagrange}
\frac{d}{dt}\,\sC_\beta(E_t) \Big|_{t = 0} = \int_{\Gamma}  \kappa_E^\Phi \,X\cdot \nu_E\,d\cH^{n-1} + \int_{\p\Omega\cap \p\Gamma} X\cdot \Big[\cR_{\pi/2}(\pi(\nabla\Phi(\nu_\Gamma)) ) + \beta \bn\Big]\,d\cH^{n-2},
\end{align}
where $\Gamma:=\pOmega E,$ $\nu_\Gamma$ is the outer unit normal to $\Gamma,$ $\bn$ is the conormal of $\Gamma$ at its boundary points  (i.e., tangent vector to $\Gamma,$ but normal to $\p\Gamma$), $\pi$ is the orthogonal projection onto the hyperplane $T$ spanned to $\{\nu_\Gamma,\bn\}$ (both defined at $\p\Gamma$), $\cR_\theta$ is the counterclockwise rotation in $T$ by angle  $\theta.$ 

Recall that by \cite[Lemma 2.1]{BKh:2018} for any $E\in \cS$
\begin{equation}\label{neat_locality}
\chi_E\in L^1(\p\Omega) \quad\text{and}\quad  E\in \cS.
\end{equation}
By \eqref{neat_locality} we can rewrite the anisotropic capillary functional \eqref{def:capil_funcos} as
$$
\sC_\beta(F) = P_\Phi(F) + \int_{\p\Omega} (\beta-\Phi^o(\be_n))\chi_F\,d\cH^{n-1}
$$
Moreover, since $G\chi_G\in L^1(\Omega)$ for any $G\in \cS,$ up to an additive constant independent of $E$ we can write
$$
\sF_{\beta,f}(E;E_0,\tau,k):=\sC_\beta(E) + 
\tfrac{1}{\tau}\int_E\sd_{E_0}dx +
\tfrac{1}{\tau}\int_{k\tau}^{(k+1)\tau}\int_E f(s,x)dxds,\quad k\ge1.
$$

\subsection{Proof of Theorem \ref{teo:intro_existGMM}}
 
For the convenience of the reader we divide the proof into smaller steps. In each step we highlight which of the assumptions on $f$ (mentioned in Theorem \ref{teo:intro_existGMM}) will be used in that step.

\subsubsection{Existence of minimizers}

Given $E_0\in \cS,$ $\tau>0$ and $k\in\N,$ let $\{E_i\}$ be a minimizing sequence of $\sF_{\beta,f}(\cdot;E_0,\tau,k).$ We may assume that $\sF_{\beta,f}(E_i;E_0,\tau,k)\le \sF_{\beta,f}(E_0\cap B_R; E_0,\tau,k)$ for some $R>0$ and for all $i\ge1$ (we need such a truncation with $B_R(0)$ because a priori $E_0$ is not bounded, and thus, in general the integral $\int_k^{k+1}ds \int_{E_0} f^+(\tau s,x)dx$ need not to be finite). Then 
\begin{multline}\label{moskav_jangi}
\sC_\beta(E_i) + \frac{1}{\tau}\int_{E_i\setminus E_0}\d_{E_0}dx + \frac{1}{\tau}\int_{k\tau}^{(k+1)\tau}ds\int_{E_i}f^+(\tau s,x)dx \\
\le \sF_{\beta,f}(E_0\cap B_R;E_0,\tau,k) + \frac{1}{\tau}\int_{E_0}\d_{E_0}dx + \frac{1}{\tau}\int_{k\tau}^{(k+1)\tau}ds\int_{\R^n}f^-(\tau s,x)dx:=C_1.
\end{multline}
In particular, by \eqref{eq:coersive_capillar} $\{P_\Phi(E_i)\}$ is bounded, and hence, by $L_\loc^1(\Omega)$-compactness of $\cS$, there exists $E_\infty\in BV_\loc(\Omega;\{0,1\})$ such that, up to a relabelled subsequence,
$E_i\to E$ in $L_\loc^1(\Omega)$ as $i\to+\infty.$ Moreover, for any bounded $U\subset\Omega$
$$
P_\Phi(E,U)\le \liminf\limits_{i\to+\infty} P_\Phi(E_i,U) \le \liminf\limits_{i\to+\infty} P_\Phi(E_i) \le \frac{1}{\eta}\,\sup_i \sC_\beta(E_i)\le \frac{C_1}{\eta}.
$$
Thus, letting $U\nearrow \R^n$ we get $P_\Phi(E)<+\infty.$ Moreover, by the isoperimetric inequality
\begin{equation}\label{isoperimetric_ineq}
P_\Phi(E) \ge c_{\Phi,n} |E|^{\frac{n-1}{n}},\quad c_{\Phi,n} = \tfrac{P_\Phi(W^\Phi)}{|W^\Phi|^{\frac{n-1}{n}}},
\end{equation}
for any bounded $U\subset\R^n$ we have
$$
|U\cap E| = \lim\limits_{i\to+\infty} |U\cap E_i| \le \liminf\limits_{i\to+\infty} |E_i| \le {c_{\Phi,n}^{-\frac{n}{n-1}}}\liminf\limits_{i\to+\infty} P_\Phi(E_i)^{\frac{n}{n-1}} \le \Big(\frac{C_1}{c_{\Phi,n}\eta}\Big)^{\frac{n}{n-1}}.
$$
Thus, $|E|<+\infty,$ i.e. $E\in \cS.$ Then the $L_\loc^1(\Omega)$-lower semicontinuity of $\sF_{\beta,f}$ implies $E$ is a minimizer.

Notice that if $E_\tau$ is a minimizer, then  as in \eqref{moskav_jangi}
\begin{equation*}
\sC_\beta(E_\tau) +\frac{1}{\tau} \int_{E_\tau \setminus E_0}\d_{E_0}dx + \frac{1}\tau \int_{k\tau}^{(k+1)\tau}ds\int_{E_\tau}f^+( s,x)dx \le C_1,
\end{equation*}
and hence, $f^+\in L^1([k\tau,(k+1)\tau]\times E_\tau).$

\subsubsection{Density estimates for minimizers}

In this section besides \eqref{hyp:1}, we assume \eqref{hyp:2}.

\begin{theorem}\label{teo:density_ests}
Let $E_0\in\cS,$ $\tau>0,$ $k\in\N$ and $E_\tau$ be a minimizer of $\sF_{\beta,f}(\cdot;E_0,\tau,k).$ Let $T>(k+1)\tau.$ Then there exists $\theta\in(0,8^{-n})$ depending only on $n,\Phi$ and $\eta$ (see \ref{eq:coersive_capillar}) such that 
\begin{equation}\label{Linfty_estimos}
\sup_{x\in \cl{E_\tau \Delta E_0}}\,\,\,\d_{E_0}(x) \le \tfrac{\sqrt\tau}{\theta}
\end{equation}
provided  $\theta\sqrt\tau\le\gamma_T.$ Moreover, if $x\in \cl{\p E_\tau}$ and $r\in(0, \theta\sqrt\tau],$ then 
\begin{gather}
\theta \le \frac{|B_r(x)\cap E_\tau|}{|B_r(x)|}  \le 1 - \theta, \label{volume_est}\\[2mm]
\frac{P(E_\tau,B_r(x))}{r^{n-1}} \ge \theta, \label{perimeter_est}
\end{gather}
where  $\gamma_T>0$ is given by \eqref{hyp:2}.
\end{theorem}

In what follows we refer to \eqref{volume_est} and \eqref{perimeter_est} as the uniform density estimates for $E_\tau.$

\begin{proof}
For shortness, we write
$$
h(\cdot):=\int_k^{k+1}f(\tau s,\cdot)ds.
$$
Let us establish \eqref{Linfty_estimos}. For each $x\in E\Delta E_0$ let $p_x\in\cl{ \pOmega E_0 }$ be such that $r_x:=|p_x-x|=\d_{E_0}(x).$ By the $1$-lipschitzianity of $\d_{E_0},$ 
$$
|\d_{E_0}(x)| \le |\d_{E_0}(p_x)| + |x-p_x| = r_x.
$$
Thus, we need to estimate $r_x.$
Fix $r\in (0,r_x)$ and set $B_r:=B_r(x).$ 

Let $x\in E\setminus E_0$ so that 
$
\sd_{E_0} \ge r_x-r
$ 
in $B_r(x).$ Then for a.e. $r\in(0,r_x)$ with $\cH^{n-1}(\p^*E_\tau\cap\p B_r)=0$ summing the equalities
\begin{align*}
&
\begin{aligned}
P_\Phi(E_\tau\setminus B_r, \Omega) - P_\Phi(E_\tau,\Omega) = & \int_{E_\tau \cap \p B_r} \Phi(\nu_{B_r})d\cH^{n-1} - P_\Phi(E_\tau,\Omega\cap B_r)\\
= &2 \int_{E_\tau\cap \p B_r} \Phi(\nu_{B_r})d\cH^{n-1} - P_\Phi(E_\tau\cap B_r,\Omega),
\end{aligned}
\\
& \int_{\p\Omega}\beta \chi_{E_\tau\setminus B_r} d\cH^{n-1} - \int_{\p\Omega}\beta \chi_{E_\tau} d\cH^{n-1} = -\int_{\p\Omega} \beta \chi_{E_\tau\cap B_r} d\cH^{n-1}, \\
& \int_{E_\tau\setminus B_r} \sd_{E_0}dx - \int_{E_\tau} \sd_{E_0}dx = - \int_{E_\tau\cap B_r } \sd_{E_0}dx,\\
& \int_{E_\tau\setminus B_r} hdx - \int_{E_\tau} hdx = - \int_{E_\tau\cap B_r } hdx
\end{align*}
and using the minimimality of $E$ we find 
\begin{multline*}
0\le \sF_{\beta,f}(E_\tau\setminus B_r) - \sF_{\beta,f}(E_\tau) =
2 \int_{E_\tau\cap \p B_r} \Phi(\nu_{B_r})d\cH^{n-1} - P_\Phi(E_\tau\cap B_r,\Omega) \\
-\int_{\p\Omega} \beta \chi_{E_\tau\cap B_r} d\cH^{n-1} - \frac{1}{\tau} \int_{E_\tau\cap B_r } \sd_{E_0}dx
-\int_{E_\tau\cap B_r } hdx.
\end{multline*}
Thus,
\begin{equation}\label{uzs671s}
2 \int_{E_\tau\cap \p B_r} \Phi(\nu_{B_r})d\cH^{n-1} \ge 
\sC_\beta(E_\tau\cap B_r) + 
\frac{r_x-r}{\tau}|E_\tau\cap B_r| +
\int_{E_\tau \cap B_r} hdx.
\end{equation}
By \eqref{eq:coersive_capillar}, the definition of the $\Phi$-perimeter, \eqref{norm_bounds} and the Euclidean isoperimetric inequality
$$
\sC_\beta(E_\tau\cap B_r)  \ge \eta P_\Phi(E_\tau\cap B_r) \ge \eta c_\Phi n\omega_n^{1/n}|E_\tau\cap B_r|^{\frac{n-1}{n}}.
$$
On the other hand, if $r\le \gamma_T,$ then by \eqref{hyp:2}
$$
\int_{E_\tau\cap B_r} |h|dx \le \frac{c_\Phi \eta n\omega_n^{1/n}}{4}\,|E_\tau\cap B_r|^{\frac{n-1}{n}},
$$
and therefore, by \eqref{uzs671s} and \eqref{norm_bounds}
$$
\frac{3c_\Phi \eta n\omega_n^{1/n}}{4}\,|E_\tau\cap B_r|^{\frac{n-1}{n}} \le 2C_\Phi \cH^{n-1}(E_\tau\cap \p B_r).
$$
Integrating this inequality we get 
\begin{equation}\label{lower_mamaos}
|E_\tau\cap B_r| \ge \Big(\frac{3c_\Phi \eta}{8C_\Phi}\Big)^n\omega_nr^n,\quad r\in [0,\gamma_T\wedge r_x].
\end{equation}
Inserting this in \eqref{uzs671s} we get 
$$
\frac{r_x-r}{\tau}\,\Big(\frac{3c_\Phi \eta}{8C_\Phi}\Big)^n\omega_nr^n \le 2C_\Phi n\omega_n r^{n-1},
$$
and therefore, 
\begin{equation}\label{biblioteka}
r_x \le g(r):=r + \frac{C_2\tau}{r},\quad r\in (0,\gamma_T\wedge r_x],
\end{equation}
where 
$$
C_2:=2C_\Phi n \Big(\frac{8C_\Phi}{3c_\Phi \eta}\Big)^n.
$$

On the other hand, if $x\in E_0\setminus E_\tau,$ then using 
\begin{multline*}
0\le \sF_{\beta,f}(E_\tau\cup B_r) - \sF_{\beta,f}(E_\tau) =
2 \int_{E_\tau^c\cap \p B_r} \Phi(\nu_{B_r})d\cH^{n-1} - P_\Phi(E_\tau^c\cap B_r,\Omega) \\
+\int_{\p\Omega} \beta \chi_{E_\tau^c\cap B_r} d\cH^{n-1} - \frac{1}{\tau} \int_{E_\tau^c\cap B_r } \sd_{E_0}dx
-\int_{E_\tau^c\cap B_r } hdx
\end{multline*}
for a.e. $r\in(0,r_x]$ with $r_x:=\dist(x,\p E_0),$ we get 
$$
2 \int_{E_\tau^c \cap \p B_r} \Phi(\nu_{B_r})d\cH^{n-1} \ge 
\sC_{-\beta}(E_\tau^c \cap B_r) + 
\frac{r_x-r}{\tau}|E_\tau^c \cap B_r| +
\int_{E_\tau^c \cap B_r} hdx,
$$
and repeating the above arguments we obtain  
\begin{equation}\label{upper_mamaos}
|E_\tau^c \cap B_r| \ge \Big(\frac{3c_\Phi \eta}{8C_\Phi}\Big)^n\omega_nr^n,\quad r\in [0,\gamma_T\wedge r_x],
\end{equation}
and hence, \eqref{biblioteka} follows.

In the remaining part of the proof we assume that $\sqrt{C_2\tau}\le \gamma_T.$ The function $g$ in \eqref{biblioteka} admits its unique global minimum at $\sqrt{C_2\tau}.$ Thus, if $r_x>\sqrt{C_2\tau},$ then $r_x\le g(\sqrt{C_2\tau}) = 2\sqrt{C_2\tau}.$ Therefore,
$$
\sup_{x\in \cl{E_\tau\Delta E_0}}\,\,\,\d_{E_0}(x) = \sup_{x\in \cl{E_\tau\Delta E_0}}\,\,\,r_x \le 2\sqrt{C_2\tau}.
$$

Now we prove density estimates. Fix any $x\in \p E_\tau,$ $r\in (0,\sqrt{C_2\tau}]$ and let $B_r:=B_r(x).$ First assume that $B_r \cap \cl{\pOmega E_\tau} = \emptyset.$ Then $B_r$ intersects only the flat part of $\p E_\tau,$ and hence, 
$$
\cH^{n-1}(B_r\cap \p E_\tau) = \omega_{n-1}r^{n-1},\quad |B_r\cap E_\tau| = \frac{\omega_n r^n}{2}.
$$
Thus, consider the case $B_r\cap \cl{\pOmega E_\tau}\ne\emptyset$ so that 
\begin{equation}\label{sdahu67}
\sup_{y\in B_r} \d_{E_0}(y) \le r + \sup_{y\in B_r\cap [E\Delta E_0]} \d_{E_0}(y) \le r + 2\sqrt{C_2\tau}.
\end{equation}
By \eqref{hyp:2}
$$
\int_{E_\tau\cap B_r} |h|dx \le \frac{c_\Phi\eta n\omega_n^{1/n}}{4}|E_\tau \cap B_r|^{\frac{n-1}{n}},
\quad
\int_{E_\tau^c\cap B_r} |h|dx \le \frac{c_\Phi\eta n\omega_n^{1/n}}{4}|E_\tau^c\cap B_r|^{\frac{n-1}{n}}.
$$
Moreover, by \eqref{sdahu67}
$$
\int_{E_\tau\cap B_r} \d_{E_0}dx \le (r+2\sqrt{C_2\tau})|E_\tau \cap B_r|^{\frac{n-1}{n}}|E_\tau\cap B_r|^{\frac{1}{n}} \le \omega_n^{1/n}r(r+2\sqrt{C_2\tau})|E_\tau \cap B_r|^{\frac{n-1}{n}}
$$
and 
$$
\int_{E_\tau^c \cap B_r} \d_{E_0}dx \le (r+2\sqrt{C_2\tau})|E_\tau^c \cap B_r|^{\frac{n-1}{n}}|E_\tau^c \cap B_r|^{\frac{1}{n}} \le \omega_n^{1/n}r(r+2\sqrt{C_2\tau})|E_\tau^c \cap B_r|^{\frac{n-1}{n}}.
$$
Thus, if we choose $r\le C_3\sqrt\tau,$ where  $C_3$ satisfies 
$$
C_3(C_3+2\sqrt{C_2}) = \frac{c_\Phi\eta n}{4},
$$
then 
$$
\int_{E_\tau\cap B_r}\Big(\frac{\d_{E_0}}{\tau} + h\Big)dx \le \frac{c_\Phi\eta n\omega_n^{1/n}}{2}|E_\tau\cap B_r|^{\frac{n-1}{n}}
$$
and 
$$
\int_{E_\tau^c\cap B_r}\Big(\frac{\d_{E_0}}{\tau} + h\Big)dx \le \frac{c_\Phi\eta n\omega_n^{1/n}}{2}|E_\tau^c \cap B_r|^{\frac{n-1}{n}}
$$
Thus, as in the proof of \eqref{lower_mamaos} and \eqref{upper_mamaos} we get 
\begin{equation}\label{asndvif76wefgz}
\Big(\frac{c_\Phi \eta}{4C_\Phi}\Big)^n\le 
\frac{|E_\tau\cap B_r|}{|B_r|} \le 1- \Big(\frac{c_\Phi \eta}{4C_\Phi}\Big)^n,\quad r\in (0,C_3\sqrt\tau].
\end{equation}
Finally, \eqref{perimeter_est} follows from \eqref{asndvif76wefgz} and the relative isoperimetric inequality for balls.
\end{proof}

From the lower perimeter density estimate in Theorem \ref{teo:density_ests} and a covering argument we get 

\begin{corollary}

Under assumptions of Theorem \ref{teo:density_ests}, any minimizer $E_\tau$ of $\sF_{\beta.f}$ satisfies 
$$
\cH^{n-1}(\cl{E_\tau} \setminus \Int{E_\tau})< +\infty \qquad\text{and}\qquad 
\cH^{n-1}(\p E_\tau \setminus \p^*E_\tau) = 0.
$$
In particular, $E_\tau$ may be assumed open.
\end{corollary}

Another corollary of density estimates is the following analogue of the volume-distance inequality of \cite{ATW:1993}.

\begin{corollary}
Let $E_0\in\cS,$ $\tau>0$ and $k\in\N$ be such that 
\begin{equation}\label{lower_ensity_chetverg}
P(E_0,B_r(x)) \ge \theta r^{n-1},\quad r\in (0,\theta\sqrt\tau],
\end{equation}
for some $\theta,\delta>0.$ Then for any $p>0$ and a minimizer $E_\tau$ of $\sF_{\beta,f}(\cdot;E_0,\tau,k)$ we have 
\begin{equation}\label{volume_inequal}
|E_\tau\Delta E_0| \le \frac{C_4}{p}\,\sC_\beta(E_0)\tau
+ 
\frac{p}{\tau}\int_{E_\tau\Delta E_0}\d_{E_0}dx
\end{equation}
provided $\tau<\theta^2 p^2,$ where 
$$
C_4:= \frac{5^n\omega_n}{c_\Phi \theta \eta}
$$
and $\eta>0$ is given in \eqref{eq:coersive_capillar}.
\end{corollary}

Specific choices of $p$ will be made in the proof of the almost-continuity of flat flows in the next section.

\begin{proof}
Let 
$$
A:=\Big\{x\in E_\tau\Delta E_0:\,\, \d_{E_0}(x)< \frac{\tau}{p}\Big\},
\quad 
B:=\Big\{x\in E_\tau\Delta E_0:\,\, \d_{E_0}(x)\ge  \frac{\tau}{p}\Big\}
$$
By the Chebyshev inequality 
$$
|B| \le \frac{p}{\tau}\int_{E_\tau\Delta E_0}\d_{E_0}dx.
$$
The set $A$ can be covered by balls $\{B_{\tau/p}(x)\}_{x\in\pOmega E_0}.$ By the Vitali covering lemma there exists an at most countable disjoint family $\{B_{\tau/p}(x_i)\}_{i\ge1}$ such that $A$ is still covered by $\{B_{5\tau/p}(x_i)\}_{i\ge1}.$ Then by \eqref{lower_ensity_chetverg} applied with $\tau/p\in (0,\theta\sqrt\tau)$ we have
\begin{align*}
|A| \le \sum_{i\ge1} |B_{5\tau/p}(x_i)| \le \frac{5^n\omega_n \tau}{p}\sum_{i\ge1}\Big( \frac{\tau}{p}\Big)^{n-1} \le \frac{5^n\omega_n \tau}{p\theta} \sum_{i\ge1} P(E,B_{\tau/p}(x_i)) \le 
\frac{5^n\omega_n \tau}{c_\Phi p\theta} P_\Phi(E).
\end{align*}
Now using \eqref{eq:coersive_capillar} and  the equality $|E_\tau\Delta E_0|=|A|+|B|$ we get \eqref{volume_inequal}.
\end{proof}

\subsubsection{Flat flows}
In this section besides \eqref{hyp:1}-\eqref{hyp:2}, we assume \eqref{hyp:f0}. Some further conditions on $f$ will be assumed later.

Notice that under assumption \eqref{hyp:1} for any $\tau>0$ and $E_0\in \cS$ we can define a flat flow $\{E(\tau,k)\}$ starting from $E_0.$ By Theorem \ref{teo:density_ests} $E(\tau,k)$ for $k\ge1$ satisfies the uniform lower perimeter estimates, and thus,   by \eqref{volume_inequal} for any $p>0$ and $1\le m_1<m_2$
\begin{multline}\label{volume_estoes1092}
|E(\tau,m_1)\Delta E(\tau,m_2)| \le\sum_{k=m_1+1}^{m_2} |E(\tau,k)\Delta E(\tau,k+1)| \\
\le 
\frac{C_4}{p} \sum_{k=m_1+1}^{m_2} \sC_\beta(E(\tau,k-1)) \tau + \frac{p}{\tau} \sum_{k=m_1+1}^{m_2} \int_{E(\tau,k-1)\Delta E(\tau,k)} \d_{E(\tau,k-1)}dx
\end{multline}
whenever $\tau<\theta^2p^2.$ Further we will estimate both sums separately.

By the minimality of $E(\tau,k)$ and \eqref{hyp:f0} for $k\ge1,$
\begin{multline}\label{minimal_flats}
\sC_\beta(E(\tau,k)) + \frac{1}{\tau}\int_{k\tau}^{(k+1)\tau}ds \int_{E(\tau,k)} f(s,x)dx + \frac{1}{\tau} \int_{E(\tau,k-1)\Delta E(\tau,k)} \d_{E(\tau,k-1)}dx \\
\le  \sC_\beta(E(\tau,k-1)) + \frac{1}{\tau}\int_{k\tau}^{(k+1)\tau}ds \int_{E(\tau,k-1)} f( s,x)dx.
\end{multline}

To estimate the differences of forcing terms we need some extra regularity conditions on $f.$ Further we fix $T>0$ and let $\tau$ be so small that $T>10\tau$ and $\frac{1}{\tau}\int_0^{2\tau}ds \int_{E_0}|f|dx$ is uniformly bounded (by assumption \eqref{hyp:f0}).  

\subsubsection*{Condition 1: $f$ bounded}

Assume \eqref{hyp:f1}.
Then applying \eqref{volume_inequal} with $p=\frac{1}{2(1+c_T)}$ we get 
\begin{multline*}
\frac{1}{\tau}  \int_{k\tau}^{(k+1)\tau} ds  \Big[\int_{E(\tau,k-1)} f( s,x)dx - \int_{E(\tau,k)} f( s,x)dx\Big]
\le  c_T |E(\tau,k-1)\Delta E(\tau,k)|\\
\le  2C_4c_T \sC_\beta(E(\tau,k-1)) \tau + 
\frac{1}{2\tau}  \int_{E(\tau,k-1)\Delta E(\tau,k)} \d_{E(\tau,k-1)}dx
\end{multline*}
provided $\tau<\theta^2/(2(1+c_T))^2.$ Inserting this estimate in \eqref{minimal_flats} we obtain 
\begin{equation}\label{zaryad_tamom_buldi}
\sC_\beta(E(\tau,k)) + \frac{1}{2\tau}\int_{E(\tau,k-1)\Delta E(\tau,k)} \d_{E(\tau,k-1)}dx 
\le (1+2C_4c_T\tau) \sC_\beta(E(\tau,k-1)).
\end{equation}
By induction 
$$
\sC_\beta(E(\tau,k)) \le (1+2C_4c_T\tau)^{k-1} \sC_\beta(E(\tau,1)),\quad k\ge1.
$$
Assuming $k<\intpart{T/\tau}$ and using elementary inequality 
$$
(1 + 2C_4c_T\tau)^{\intpart{T/\tau}-1} = \Big((1 + 2C_4c_T\tau)^{\frac{1}{2C_4c_T\tau}}\Big)^{2C_4c_T\tau(\intpart{T/\tau}-1)}  \le e^{2C_4c_TT}.
$$
we deduce 
\begin{equation}\label{gako_ubil_on}
\sC_\beta(E(\tau,k)) \le e^{2C_4c_TT}\sC_\beta(E(\tau,1)),\quad k=1,\ldots,\intpart{T/\tau}-1. 
\end{equation}
Moreover, given $1<m_1<m_2<\intpart{T/\tau}$, summing \eqref{zaryad_tamom_buldi}
in $k=m_1+1,\ldots,m_2$ we get 
\begin{align}
\frac{1}{2\tau}\sum_{k=m_1+1}^{m_2}   & \int_{E(\tau,k-1)\Delta E(\tau,k)} \d_{E(\tau,k-1)}dx \nonumber \\
\le &  
\sC_\beta(E(\tau,m_1))-\sC_\beta(E(\tau,m_2)) + 2C_4c_T\sum_{k=m_1+1}^{m_2}\sC_\beta(E(\tau,k-1))\tau \nonumber \\
\le & e^{2C_4c_TT}\sC_\beta(E(\tau,1)) + 2C_4c_T e^{2C_4c_TT} \sC_\beta(E(\tau,1)) (m_2-m_1)\tau, \label{allo_kimbu}
\end{align}
where in the last inequality we used \eqref{gako_ubil_on}.

Next fix $0<s<t<T$ and let $\tau>0$ be so small that $s>10\tau$ and $t-s>10\tau.$ Applying \eqref{volume_estoes1092} with $m_1=\intpart{s/\tau},$  $m_2=\intpart{t/\tau}$ and $p=|t-s|^{1/2},$ and using \eqref{gako_ubil_on} and \eqref{allo_kimbu} we get 
\begin{multline*}|E(\tau,\intpart{s/\tau})\Delta E(\tau,\intpart{t/\tau})| \le  \frac{C_4e^{2C_4c_TT} \sC_\beta(E(\tau,1))  (t-s+\tau)}{|t-s|^{1/2}} \\
+ 2\Big(e^{2C_4c_TT}\sC_\beta(E(\tau,1)) + 2C_4c_T e^{2C_4c_TT}\sC_\beta(E(\tau,1))  (t-s+\tau)\Big)|t-s|^{1/2},
\end{multline*}
and therefore,
\begin{equation}\label{holderian00}
|E(\tau,\intpart{s/\tau})\Delta E(\tau,\intpart{t/\tau})| \le
C_5 \sC_\beta(E(\tau,1)) \Big(|t-s|^{1/2}+|t-s|^{3/2} + \tfrac{\tau(|t-s|+1)}{|t-s|^{1/2}}\Big),
\end{equation}
where 
$$
C_5:=(C_4 + 2 + 4C_4c_TT) e^{2C_4c_TT}.
$$
It remains to estimate $\sC_\beta(E(\tau,1))$ uniform in $\tau.$ Applying  \eqref{minimal_flats} with $k=1$ we get 
\begin{align*}
\sC_\beta(E(\tau,1)) \le & \sC_\beta(E_0) + \frac{1}{\tau}\int_\tau^{2\tau} ds \int_{E_0}f(s,x)dx - \frac{1}{\tau}\int_{\tau}^{2\tau} ds\int_{E_\tau}f(s,x)dx\\
\le & \sC_\beta(E_0) + \frac{2}{\tau}\int_0^{2\tau} ds\int_{\R^n}f(s,x)dx := c_\tau,
\end{align*}
where by assumption \eqref{hyp:f1} $c_\tau$ is uniformly bounded as $\tau\to0^+.$ Owing this and \eqref{holderian00}, and repeating the standard arguments in the existence of GMM (see e.g. \cite{BKh:2018}) we conclude  $GMM(\sF_{\beta,f},E_0)\ne\emptyset$ and each GMM is locally $1/2$-H\"older continuous in time.

\subsubsection*{Condition 2: $f$ locally time-Lipschitz}

Assume \eqref{hyp:f2} and set
$$
\sigma_k:=\sC_\beta(E(\tau,k))+ \frac{1}{\tau}\int_{k\tau}^{(k+1)\tau}ds \int_{E(\tau,k)} f(s,x)dx,\quad k\ge0.
$$
By \eqref{hyp:f2} for all $1\le k\le \intpart{T/\tau}-1$ we have
\begin{multline*}
\int_{k\tau}^{(k+1)\tau}ds \int_{E(\tau,k)} f(s,x)dx - \int_{(k-1)\tau}^{k\tau}ds \int_{E(\tau,k-1)} f(s,x)dx\\
\le \int_{k\tau}^{(k+1)\tau}ds \int_{E(\tau,k-1)} |f(s,x) - f(s+\tau,x)|dx \le c_T\tau^2.
\end{multline*}
Therefore, by \eqref{minimal_flats} 
$$
\sigma_k + \frac{1}{\tau}\int_{E(\tau,k-1)\Delta E(\tau,k)} \d_{E(\tau,i-1)}dx \le \sigma_{k-1} + c_T\tau
$$
and summing these inequalities
\begin{equation*}
\sigma_{k} + \frac{1}{\tau}\sum_{i=1}^k \int_{E(\tau,i-1)\Delta E(\tau,i)} \d_{E(\tau,i-1)}dx  \le \sigma_0 + c_T k\tau .
\end{equation*}
Let us rewrite this inequality as
\begin{multline}\label{further_estimate9i}
\sC_\beta(E(\tau,k)) + \frac{1}{\tau}\sum_{i=1}^k \int_{E(\tau,i-1)\Delta E(\tau,i)} \d_{E(\tau,i-1)}dx 
\\ 
\le \sigma_0 + c_T k\tau - \frac{1}{\tau}\int_{k\tau}^{(k+1)\tau}ds \int_{E(\tau,k)}f(s,x)dx.
\end{multline}
By \eqref{hyp:f2}
\begin{multline*}
\Big|\frac{1}{\tau}\int_{k\tau}^{(k+1)\tau}ds \int_{E(\tau,k)}f(s,x)dx-\frac{1}{\tau}\int_{0}^{\tau}ds \int_{E(\tau,k)}f(s,x)dx \Big|\\
\le 
\sum_{i=1}^k \int_i^{i+1}ds \int_{\R^n}|f(s\tau,x)-f(s\tau+\tau,x)|\le c_Tk\tau.
\end{multline*}
Inserting this in \eqref{further_estimate9i}, for any  $k<\intpart{T/\tau}$ we get 
\begin{multline}\label{furo0o}
\sC_\beta(E(\tau,k)) + \frac{1}{\tau}\sum_{i=1}^k \int_{E(\tau,i-1)\Delta E(\tau,i)} \d_{E(\tau,i-1)}dx 
\\ 
\le \sC_\beta(E_0) + 2c_T T + \frac{2}{\tau}\int_{0}^{2\tau}ds \int_{\R^n}|f(s,x)|dx=:c_\tau',
\end{multline}
where $c_\tau'$ is uniformly bounded for small $\tau.$
Now take any $0<s<t<T$ and let $\tau$ be so small that $t-s>10\tau$ and $s>10\tau.$ Applying \eqref{volume_estoes1092} with $m_1:=\intpart{s/\tau},$  $m_2:=\intpart{t/\tau}$ and $p=|t-s|^{1/2},$ and employing \eqref{furo0o} we obtain 
$$
|E(\tau,\intpart{s/\tau})\Delta E(\tau,\intpart{t/\tau})| \le (C_4+1)c_\tau'\Big(|t-s|^{1/2}+ \tfrac{\tau}{|t-s|^{1/2}}
\Big).
$$
This implies $GMM(\sF_{\beta,f},E_0)\ne\emptyset$ for any $E_0\in \cS$ and each GMM is $1/2$-H\"older continuous in time.

\subsection{Uniform boundedness of GMM}\label{subsec:unifbounded_gmm}

In this section we obtain $L^\infty$-bounds for GMM  starting from a bounded set $E_0$, assuming the growth condition \eqref{hyp:fb2}.

Recall that in the literature (see e.g. \cite{ATW:1993,ChDgM:2023,MSS:2016}) without boundary conditions  the following comparison can be established: \it if $F_0\subset W_{r_0}^\Phi$ and $F_\tau$ of the standard Almgren-Taylor-Wang functional with forcing $f$ satisfies $F_\tau\subset W_{r_\tau}^\Phi,$ then: 
\begin{itemize}
 \item if $f\equiv0,$ then (by truncation with convex sets \cite{ATW:1993,LS:1995}) $r_\tau=r_0,$
 
 \item if $f\ne0,$ then (by trucantion with balls \cite{ChDgM:2023})
$$
\frac{r_\tau - r_0}{\tau} \le c +\frac{c}{r_\tau}
$$
for some constant $c>0.$  
\end{itemize}
\rm
Below we establish similar comparison principle, but due to the  boundary term, we cannot apply Wulff shapes. Rather, we use Winterbottom shapes \cite{Kholmatov:2024,KSch:2024,Maggi:2012}: given a constant $\beta_0\in (-\Phi(\be_n),\Phi(\be_n)),$ the part of the Wulff shape $W_{\beta_0,R}:=\Omega\cap W_R^\Phi(\beta_0R\be_n)$ centered at $\beta_0R\be_n$ of radius $R,$ the so-called \emph{Winterbottom shape}, satisfies 
\begin{equation}\label{winterbot_ineq9i8}
\sC_{\beta_0}(E) \ge c_{\Phi,\beta_0,n}\,|E|^{\frac{n-1}{n}},\quad c_{\Phi,\beta_0,n}:=\tfrac{\sC_{\beta_0}(W_{\beta_0,R})}{|W_{\beta_0,R}|^{\frac{n-1}{n}}},
\end{equation}
for all $E\in \cS.$ Note that the isoperimetric constant $c_{\Phi,\beta_0,n}$ independent of $R$ and horizontal translations of the Wuinterbottom shape.
Recall that without forcing, in \cite{BKh:2018} we used a sort of ``mean convex'' sets (for capillary functional) to bound uniformly the minimizers of $\sF_\beta.$

\begin{lemma}\label{lem:winterbottom}

Let $\tau>0,$ $k\in\N, $ $F_0\in\cS$ be a bounded set and $F_\tau$ be a minimizer of $\sF_{\beta,f}(\cdot;F_0,\tau,k),$ which is bounded in view of \eqref{Linfty_estimos}. Let $-\Phi(\be_n) < \beta_0 < -(1-2\eta)\Phi(\be_n)$ be a constant and let $W_{\beta_0,r_0}$ contain $F_0$ and $W_{\beta_0,r_\tau}$  be the smallest Winterbottom shape containing $F_\tau.$ Then either $r_\tau\le r_0,$ where $R_0$ is given in \eqref{hyp:fb2}, or $r_\tau>r_0$ and
$$
r_\tau \le (1 + C_6\tau)r_0 + C_7\tau
$$
for some $C_6,C_7$ depending only on $\beta_0$ and $T,$ and for all $\tau<\frac{1}{4C_6}.$
\end{lemma}

\begin{figure}[htp!]
\includegraphics[width=0.6\textwidth]{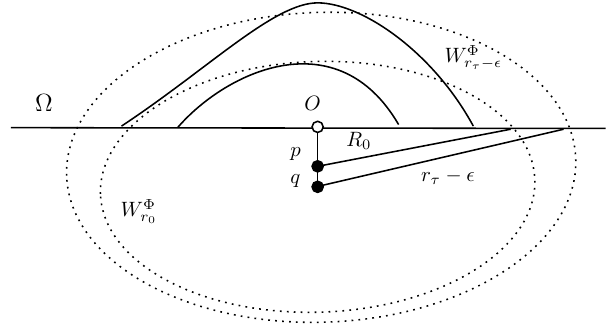}
\caption{\small Comparison with Winterbottom shapes.}\label{fig:winterbottom}
\end{figure}

\begin{proof}
Note that for each $r>0$ there exists a unique Winterbottom shape $W_{\beta_0,r}$ whose center lies on the vertical line passing through the origin (see Figure \ref{fig:winterbottom}), and  $W_{\beta_0,r'}\subset W_{\beta_0,r''}$ whenever $r'<r''.$ Therefore, if $r_\tau\le r_0$ we are done. Otherwise fix $\epsilon\in(0,r_\tau-r_0)$ and consider the Winterbottom shape $W_{\beta_0,r_\tau-\epsilon}.$ By the minimality of $r_\tau,$ 
$$
|F_\tau\setminus W_{\beta_0,r_\tau-\epsilon}|\searrow0\quad  \text{as $\epsilon\to0^+.$}
$$ 
Let us estimate 
\begin{align*}
0 \le & \sF_{\beta,f}(W_{\beta_0,r_\tau-\epsilon}\cap F_\tau;F_0,\tau,k) -  \sF_{\beta,f}(F_\tau;F_0,\tau,k) \\
= & \sC_\beta(W_{\beta_0,r_\tau-\epsilon}\cap F_\tau) - \sC_\beta(F_\tau) - \frac{1}{\tau}\int_{F_\tau\setminus W_{\beta_0,r_\tau-\epsilon}} \sd_{F_0}dx - \frac{1}{\tau} \int_{k\tau}^{(k+1)\tau}ds \int_{F_\tau\setminus W_{\beta_0,r_\tau-\epsilon}}fdx
\\
= :& I_1 - I_2 - I_3.
\end{align*}
Since 
$$
\sd_{E_0}=\d_{E_0}\ge c_\Phi\d_{E_0}^{\Phi^o} \ge c_\Phi\beta_0 (r_\tau-r_0-\epsilon)\quad \text{in $F_\tau\setminus W_{\beta_0,r_\tau-\epsilon},$}
$$
and, recalling $r_\tau>R_0,$ by \eqref{hyp:fb2} and \eqref{norm_bounds}
$$
|f(s,x)|\le a_T+b_T|x| \le a_T + b_T(|x+\beta_0 (r_\tau-\epsilon) \be_n|+|\beta_0 (r_\tau-\epsilon) \be_n|) \le a_T +\Big(b_T\beta_0+\tfrac{b_T}{c_\Phi}\Big)(r_\tau-\epsilon)
$$
for any $s\in [k\tau,(k+1)\tau]$ and $x\in F_\tau\setminus W_{\beta_0,r_\tau-\epsilon}.$ Therefore, 
$$
I_2\ge \frac{c_\Phi\beta_0(r_\tau-r_0-\epsilon)}{\tau}|F_\tau\setminus W_{\beta_0,r_\tau-\epsilon}|
\quad\text{and}\quad 
|I_3|\le (a_T+c_Tr_\tau)|F_\tau\setminus W_{\beta_0,r_\tau-\epsilon}|.
$$
Moreover, for a.e. $\epsilon$ using \eqref{winterbot_ineq9i8} we get 
\begin{align*}
I_1 = & \sC_{\beta_0}(W_{\beta_0,r_\tau-\epsilon}) - \sC_{\beta_0}(F_\tau\cup W_{\beta_0,r_\tau-\epsilon}) + \int_{\p\Omega} (\beta_0-\beta)\chi_{F_\tau\cup W_{\beta_0,r_\tau-\epsilon}}d\cH^{n-1}\\
\le & c_{\Phi,\beta_0,n} \Big(|W_{\beta_0,r_\tau-\epsilon}|^{\frac{n-1}{n}} - |F_\tau\cup W_{\beta_0,r_\tau-\epsilon}|^{\frac{n-1}{n}}\Big)\le0.
\end{align*}
Since $I_2\le I_1 + |I_3|,$ from these estimates we deduce 
$$
\frac{r_\tau - r_0 -\epsilon}{\tau} \le \frac{a_T+b_Tr_\tau}{c_\Phi\beta_0}.
$$
Thus, letting $\epsilon\to0^+$ we get 
$$
r_\tau \le \frac{r_0}{1 - \frac{b_T}{c_\Phi\beta_0}\tau} + \frac{a_T\tau}{c_\Phi\beta_0(1-\frac{b_T}{c_\Phi\beta_0}\tau)}
$$
provided $b_T\tau<c_\Phi\beta_0$. This implies the thesis with suitable $C_6$ and $C_7$ depending only on $a_T,$ $b_T$ and $\beta_0.$
\end{proof}

Now consider any flat flow $\{E(\tau,k)\}$ starting from $E_0$ and let $W_{\beta_0,r(\tau,k)}$ be Winterbottom shapes such that contaning $E(\tau,k)$ such that $k\mapsto r(\tau,k)$ is nondecreasing and $r(\tau,0)=r_0>R_0$. By Lemma \ref{lem:winterbottom} for each $k\ge1$ we may assume either $r(\tau,k)= r(\tau,k-1)$ or 
$$
r(\tau,k-1)<r(\tau,k) < (1+C_6\tau)r(\tau,k-1)+C_7\tau.
$$  
Then applying induction argument we find 
$$
r(\tau,k) \le (1+C_6\tau)^kr_0 + C_7\tau \frac{(1+C_6\tau)^k-1}{(1+C_6\tau)-1}<\Big(r_0 + \tfrac{C_7}{C_6}\Big)(1+C_6\tau)^k.
$$
Thus, if $k\le \intpart{T/\tau},$ then 
$$
(1+C_6\tau)^k\le \Big((1+C_6\tau)^{\frac{1}{C_6\tau}}\Big)^{C_6\tau \intpart{T/\tau}} \le e^{C_6T}.
$$
Therefore, for every $E(\cdot)\in GMM(\sF_{\beta,f},E_0)$ and $t\in [0,T)$ we get $E(t)\subset W_{\beta_0,e^{C_6T}}.$

\section{Some comparison principles}\label{sec:comparison}

In this section we establish some comparison principles as in \cite[Section 6]{BKh:2018}.

\subsection{Discrete comparison and comparison of GMMs}

We start this section with 

\begin{theorem}[\textbf{Discrete comparison principle}]\label{teo:compare_minima}
Let $\tau>0,$ $k\in\N,$ $\beta_i$ satisfy \eqref{beta_condio} and $f_i$ satisfy \eqref{hyp:1}, and $E_0^{(i)} \in \cS,$ $i=1,2.$

\begin{itemize}
\item[\rm(a)] Assume that $\beta_1\ge \beta_2$ $\cH^{n-1}$-a.e. on $\p\Omega,$ $E_0^{(1)}\subset E_0^{(2)}$ and $f_1 > f_2$ a.e. in $\R_0^+\times \Omega.$ Then for any minimizer $E_\tau^{(i)}$ of $\sF_{\beta_i,f_i}(\cdot;E_0^{(i)},\tau,k)$ 
$$
E_\tau^{(1)}\subset  E_\tau^{(2)}.
$$

\item[\rm(b)] Assume that $\beta_1\ge \beta_2$ $\cH^{n-1}$-a.e. on $\p\Omega,$ $E_0^{(1)}\prec E_0^{(2)}$ and $f_1 \ge f_2$ a.e. in $\R_0^+\times \Omega.$ Then for any minimizer $E_\tau^{(i)}$ of $\sF_{\beta_i,f_i}(\cdot;E_0^{(i)},\tau,k)$ 
$$
E_\tau^{(1)}\subset  E_\tau^{(2)}.
$$

\item[\rm(c)] Assume that $\beta_1\ge \beta_2$ $\cH^{n-1}$-a.e. on $\p\Omega,$ $E_0^{(1)}\subset E_0^{(2)}$ and $f_1 \ge f_2$ a.e. in $\R_0^+\times \Omega.$ Then there exist  minimizers $E_{\tau*}^{(1)}$ of $\sF_{\beta_1,f_1}(\cdot;E_0^{(1)},\tau,k)$ and $E_{\tau}^{(2)*}$ of $\sF_{\beta_2,f_2}(\cdot;E_0^{(2)},\tau,k)$  such that 
$$
E_{\tau*}^{(1)}\subset  E_\tau^{(2)}\quad\text{and}\quad E_\tau^{(1)}\subset  E_\tau^{(2)*}.
$$

\item[\rm (d)] If $\beta_1=\beta_2=:\beta,$ $f_1=f_2=:f$ and $E_0^{(1)}=E_0^{(2)}=:E_0,$ then there exist minimizers  $E_{\tau*}$ and $E_\tau^*$ of $\sF_{\beta,f}(\cdot;E_0,\tau,k)$ such that every minimizer $E_\tau$ satisfies
$$
E_{\tau*}\subset E_\tau\subset E_\tau^*.
$$
\end{itemize}
\end{theorem}

Setting 
$$
h_i:=\frac{1}{\tau}\sd_{E_0^{(i)}} + \frac{1}{\tau}\int_{k\tau}^{(k+1)\tau}f_i(s,\cdot)ds,\quad i=1,2,
$$
we observe that assumptions (a) and (b) resp. (c) imply $h_1>h_2$ resp. $h_1\ge h_2.$ Since
$$
\sF_{\beta_i,f_i}(E;E_0^{(i)},\tau,k) = \sC_{\beta_i}(E) + \int_Eh_idx,
$$
$\sF_{\beta_i,f_i}$ is a sort of prescribed curvature functional, for which comparison principles are well-established (see also \cite[Section 6]{BKh:2018}). Therefore, we omit the proof.

We refer to  $E_{\tau*}$ and $E_\tau^*$ as the minimal  and maximal minimizers of $\sF_{\beta,f}(\cdot;E_0,\tau,k).$

Now we are ready to establish comparison between GMMs.

\begin{proof}[Proof of Theorem \ref{teo:comparing_gmms}]

(a) Take any $E^{(2)}(\cdot) \in GMM(\sF_{\beta_2,f_2},E_0^{(2)})$  and let $\{E^{(2)}(\tau_i,k)\}$ be flat flows satisfying 
\begin{equation}\label{acnsia78}
\lim_{i\to+\infty}|E^{(2)}(\tau_i,\intpart{t/\tau_i}) \Delta E^{(2)}(t)| = 0\quad\text{for any $t\ge0,$}
\end{equation}
here $\tau_i\to0^+.$  For each $i,$ let $\{E^{(1)}(\tau_i,k)_*\}$ be a flat flow starting from $E_0^{(1)},$ consisting of minimal minimizers. By the discrete comparison principle (Theorem \ref{teo:compare_minima} (c))
\begin{equation}\label{sancuiwbv7}
E^{(1)}(\tau_i,k)_* \subset E^{(2)}(\tau_i,k) \quad \text{for any $k\ge0$}. 
\end{equation}
Repeating the same arguments of the proof of Theorem \ref{teo:intro_existGMM} we can show that there exists a subsequence $\{\tau_{i_j}\}$ and a $E^{(1)}(\cdot)_* \in GMM(\sF_{\beta_1,f_1},E_0^{(1)})$ such that 
$$
\lim_{j\to+\infty}|E^{(1)}(\tau_{i_j},\intpart{t/\tau_{i_j}})_* \Delta E^{(1)}(t)_*| = 0\quad\text{for any $t\ge0.$}
$$
This, \eqref{acnsia78} and \eqref{sancuiwbv7} imply  \eqref{minimal_gmm}. 

(b) is proven analogously using the maximal minimizers of $\sF_{\beta_2,f_2}.$
\end{proof}

\subsection{Smooth inner and outer barriers}

In this section we assume $\Phi$ is an elliptic  $C^3$ anisotropy, $\beta\in C^1(\p\Omega)$ satifies \eqref{beta_condio} and $f\in C^1(\R_0^+\times\cl{\Omega}).$ 

\begin{lemma}\label{lem:regularity_minimizers}
Let $E_0\in\cS$ be a bounded set, for $\tau>0$ and $k\in\N$ let $E_\tau$ be a (bounded)  minimizer of $\sF_{\beta,f}(\cdot;E_0,\tau,k)$ and let $\Gamma:=\cl{\pOmega E_\tau}.$ Then:
\begin{itemize}
\item[\rm(a)] there exists a closed set $\Sigma\subset \Gamma$ with $\cH^{n-3}(\Sigma)=0$ such that $\Gamma\setminus \Sigma$ is a $C^{2+\alpha}$-hypersurface with boundary; if $\Phi$ is Euclidean, then $\cH^{n-4}(\Sigma)=0;$ 

\item[\rm(b)] for every $x\in \Omega\cap (\Gamma\setminus \Sigma)$ from the first variation formula \eqref{euler_lagrange} it follows
\begin{equation*} 
\frac{1}{\tau} \sd_{E_0}(x) + \kappa_{E_\tau}(x) + \frac{1}{\tau}\int_{k\tau}^{(k+1)\tau} f(s,x)ds = 0; 
\end{equation*}

\item[\rm(c)] at every $x\in \p\Omega\cap (\Gamma\setminus \Sigma)$ the anisotropic contact angle condition holds:
$$
\nabla\Phi(\nu_{E_\tau}(x))\cdot \be_n = -\beta.
$$

\end{itemize}

\end{lemma}

\begin{proof}
The assertions (a) and (c) follows from \cite{DPhM:2015,DPhM:2017} while (b) follows from (a) and the regularity of $f$ and the first variation formula \eqref{euler_lagrange}.
\end{proof}

The main result of this section is the following analogue of \cite[Lemma 7.3]{ATW:1993} (see also \cite[Lemma 2.13]{Kholmatov:2023}).

\begin{proposition}\label{prop:atw_nonpol}
Assume either $n\le 3$ if $\Phi$ is any elliptic $C^{3}$-anisotropy or $n\le 4$ if $\Phi$ is Euclidean. Let $E_0\in\cS$ be a bounded set and for $\tau>0,$ $k\in\N,$ let $E_\tau$ be a minimizer of $\sF_{\beta,f}(\cdot;E_0,\tau,k)$ and let $G_0,G_\tau$ be bounded sets with $C^{2+\alpha}$ free boundaries $\cl{\pOmega G_0}$ and $\cl{\pOmega G_\tau}$. 

\begin{itemize}
\item[\rm (a)] Let $E_0\subset  G_0,$ $E_\tau\subset  G_\tau,$ $G_\tau$ satisfies the anisotropic contact angle condition with $\beta-s$ for some $s\in (0,\eta)$ and 
\begin{equation}\label{kuchli_qdfu}
\frac{\sd_{G_0}(x)}{\tau} + \kappa^\Phi_{G_\tau}(x) +\frac{1}{\tau}\int_{k\tau}^{(k+1)\tau} f(s,x)ds > 0 \quad   \text{on $\Omega\cap \p G_\tau.$}
\end{equation}
Then $E_\tau\prec G_\tau.$

\item[\rm (b)] Let $G_0\subset  E_0,$ $G_\tau\subset  E_\tau,$ $G_\tau$ satisfies the anisotropic contact angle condition with $\beta+s$ for some $s\in (0,\eta)$ and 
$$
\frac{\sd_{G_0}(x)}{\tau} + \kappa^\Phi_{G_\tau}(x) +\frac{1}{\tau}\int_{k\tau}^{(k+1)\tau} f(s,x)ds < 0 \quad   \text{on $\Omega\cap \p G_\tau.$}
$$
Then $G_\tau\prec E_\tau.$
\end{itemize}
\end{proposition}

\begin{proof}
(a) By the assumption on the dimension $\Omega\cap\p^*E_\tau=\pOmega E_\tau.$ Thus, there exists $x_0\in \Omega\cap \pOmega E_\tau\cap \pOmega G_\tau,$ then by assumption $E_0\subset F_0$ we get $\sd_{E_0}(x_0)\ge \sd_{F_0}(x_0)$ and by assumption $E_\tau\subset F_\tau,$ we get $\kappa_{E_\tau}^\Phi(x_0) \ge \kappa_{F_\tau}^\Phi(x_0).$ Therefore, 
$$
0 = \frac{\sd_{E_0}(x_0)}{\tau} + \kappa_{E_\tau}^\Phi(x_0) + \frac{1}{\tau}\int_{k\tau}^{(k+1)\tau} f(s,x_0)ds \ge \frac{\sd_{F_0}(x_0)}{\tau} + \kappa_{F_\tau}^\Phi(x_0) + \frac{1}{\tau}\int_{k\tau}^{(k+1)\tau} f(s,x_0)ds,
$$
which contradicts to \eqref{kuchli_qdfu}. Hence, $\Omega\cap \p E_\tau \cap \p G_\tau = \emptyset.$ Moreover, by Lemma \ref{lem:regularity_minimizers} (a)  $\pOmega E_\tau$ satisfies the anisotropic contact angle condition with $\beta$ at the boundary. Since $\pOmega G_\tau$ satisfies this condition with $\beta-s,$ we have also 
$\cl{\pOmega E_\tau}\cap \cl{\pOmega G_\tau} = \emptyset.$ This implies $E_\tau\prec G_\tau.$

(b) is proven similarly. 
\end{proof}

\subsection{Comparison of flat flows with truncated Wulff shapes}
In this section we assume that $f$ is bounded.

\begin{theorem}\label{teo:compare_with_ball}
Let $E_0\in\cS,$ $\beta$ satisfy \eqref{beta_condio} and $p\in \Omega$ with $R_0:=\dist(p,\pOmega E_0)>0.$ For $\tau>0$ let $\{E(\tau,k)\}$ be flat flows starting from $E_0.$ Then for any $\beta_0\in (\|\beta\|_\infty\},1)$  
\begin{align}
\Omega\cap W_{R_0}^\Phi(p)\subset E_0 \qquad & \Longrightarrow \qquad \Omega\cap W_{\frac{\beta_0 R_0}{16\Phi(\be_n)}}^\Phi(p)\subset E(\tau,k), \label{stay_inside}\\
W_{R_0}^\Phi(p)\cap E_0 = \emptyset \qquad & \Longrightarrow \qquad W_{\frac{\beta_0 R_0}{16\Phi(\be_n)}}^\Phi(p)\cap E(\tau,k) = \emptyset \label{stay_outside}
\end{align}
whenever $0<\tau < \vartheta_0 R_0^2$ and $0\le k\tau \le \frac{\vartheta_0R_0^2}{R_0+1},$ where $\vartheta_0\in(0,1)$ is a constant depending only on $\beta_0$ and the constant $\theta$ of Theorem \ref{teo:density_ests}.
\end{theorem}

Theorem \ref{teo:compare_with_ball} is a generalization of \cite[Theorem 5.4]{ATW:1993} and \cite[Theorem 2.11]{Kholmatov:2023} in the anisotropic capillary setting. Notice that due to the presence of boundary terms in $\sF_{\beta,f},$ we cannot argue as in the proof of \cite[Theorem 5.4]{ATW:1993}.
We postpone the proof after the following lemma.

\begin{lemma}\label{lem:compare_setg}
Let $E_0,F_0$ be bounded sets of finite perimeter in $\Omega$ with $|\p E_0|=|\p F_0|=0,$ $\beta_1,\beta_2$ satisfy \eqref{beta_condio}, $f$ is bounded and for $\tau>0$ and $k\in\N,$ let $E_\tau$ be a minimizer of $\sF_{\beta_1,f}(\cdot;E_0,\tau,k).$ 

\begin{itemize}
\item[\rm(a)] Let $F_0\subseteq E_0,$ $\beta_2\ge \beta_1$ and $F_{\tau*}$ be the minimal minimizer of $\sF_{\beta_2,f}(\cdot;F_0,\tau,k).$ Then $F_{\tau*}\subseteq E_\tau.$

\item[\rm(b)] Let $E_0\cap F_0=\emptyset,$ $\beta_1+\beta_2\ge0$ and $F_{\tau*}$ be the minimal minimizer of $\sF_{\beta_2,-f}(\cdot;F_0,\tau,k).$ Then $F_{\tau*}\cap E_\tau=\emptyset.$
\end{itemize}

\end{lemma}

\begin{proof}
(a) follows from Theorem \ref{teo:compare_minima} (c). To prove (b) we sum the inequalities
$$
\sF_{\beta_1,f}(E_\tau;E_0,\tau,k) \le \sF_{\beta_1,f}(E_\tau\setminus F_{\tau*};E_0,\tau,k), \quad 
\sF_{\beta_2,-f}(F_{\tau*};F_0,\tau,k) \le \sF_{\beta_2,-f}(F_{\tau*}\setminus E_{\tau};F_0,\tau,k)
$$
and using 
$$
P(E_\tau,\Omega) + P(F_{\tau*},\Omega) \le P(E_\tau\cap F_{\tau*}^c,\Omega) + P(F_{\tau*}^c\cup E_\tau,\Omega) = P(E_\tau\setminus F_{\tau*},\Omega) + P(F_{\tau*}\setminus E_\tau,\Omega)
$$
we get 
$$
\frac1\tau\int_{F_{\tau*} \cap E_\tau} [\sd_{F_0} + \sd_{E_0}]dx + \int_{\p\Omega} [\beta_1+\beta_2]\chi_{F_{\tau*}\cap E_\tau}d\cH^{n-1}\le0.
$$
Since $E_0\cap F_0=\emptyset$ and $|\p F_0|=|\p E_0|=0,$ we have $\sd_{E_0} + \sd_{F_0}>0$ a.e. in $\Omega.$ Therefore, recalling $\beta_1+\beta_2\ge0,$ we find that the last inequality holds if and only if $|E_\tau\cap F_{\tau*}|=0.$
\end{proof}

Now we are ready to prove relations \eqref{stay_inside}-\eqref{stay_outside}.

\begin{proof}[Proof of Theorem \ref{teo:compare_with_ball}]
Following arguments of \cite[Theorem 2.11]{Kholmatov:2023} we establish only \eqref{stay_outside}, the proof of  \eqref{stay_inside} being similar. 
We divide the proof into smaller steps. Fix $\beta_0\in ((1-2\eta)\Phi(\be_n),\Phi(\be_n)).$ Depending on the position of $p$ we distinguish three cases. 
\bigskip

{\it Case 1:} $W_{R_0}^\Phi(p)\subset \Omega.$ 
\bigskip

Before we proceed, we need some preiminaries. For shortness, set $W_r^\Phi:=W_r^\Phi(p).$ Let $F_0:=\Omega\cap W_r^\Phi$ for some $r>0$ and for $\tau>0$ and $k\in\N$ let $F_{\tau*}$ be the minimal minimizer of $\sF_{\beta_0,-f}(\cdot; F_0,\tau,k).$ Because of the forcing, in general, $F_{\tau*}$ is not necessarily a Wulff shape. By Theorem \ref{teo:density_ests}
$$
\sup_{F_{\tau*}\Delta F_0}\,\d_{F_0}\le \tfrac{\sqrt\tau}{\theta},
$$
thus, further assuming $0<\tau<\frac{c_\Phi^2\theta^2r^2}{25}$ and using \eqref{norm_bounds} and the definition of $\Phi^o$ we get 
$$
\frac{1}{c_\Phi}\, \dist(\p W_r^\Phi ,\p W_{4r/5}^\Phi )\ge  \dist_{\Phi^o}(\p W_r^\Phi ,\p W_{4r/5}^\Phi ) =\tfrac{r}{5}>\tfrac{\sqrt\tau}{c_\Phi\theta}
$$
and therefore, $W_{4r/5}^\Phi\subset  F_{\tau*}.$ Let $W_\rho^\Phi$ be the maximal Wulff shape such that $\Omega\cap W_{\rho}^\Phi \subset F_{\tau*}.$  Clearly, $\rho\ge4r/5.$ We would like to estimate $\rho$ from above.

{\it Claim 1:} Either $\rho\ge r$ or 
\begin{equation}\label{muhojirlar_til}
r> \rho \ge r- \Big(C_{8}+\tfrac{C_9}{r}\Big)\tau
\end{equation}
for some constants $C_8,C_9>0$ depending only on $\Phi,$ $n$ and $\|f\|_\infty.$

Indeed, assume that $\rho<r$ and fix any $\epsilon\in (0,r-\rho).$ By the minimality of $F_{\tau*}$ 
\begin{align*}
0\le & \sF_{\beta_0,-f}(F_{\tau*}\cup W_{\rho+\epsilon}^\Phi;F_0,\tau,k) - \sF_{\beta_0,-f}(F_{\tau*};F_0,\tau,k)  \\
= & P_\Phi(F_{\tau*}\cup W_{\rho+\epsilon}^\Phi) -P_\Phi(F_{\tau*}) 
+ \frac{1}{\tau}\int_{W_{\rho+\epsilon}^\Phi\setminus F_{\tau*}} \sd_{F_0}dx 
+ \frac{1}{\tau}\int_{k\tau}^{(k+1)\tau} ds \int_{W_{\rho+\epsilon}^\Phi\setminus F_{\tau*}} f(s,x) dx \\
= & I_1+I_2+I_3.
\end{align*}
Notice that by the maximality of $\rho,$ $|W_{\rho+\epsilon}^\Phi\setminus F_{\tau*}|\searrow0$ as $\epsilon\to0^+.$
Since $f$ is bounded, 
$$
|I_3|\le \|f\|_\infty |W_{\rho+\epsilon}^\Phi\setminus F_{\tau*}|.
$$
Moreover, by the assumption  $\rho+\epsilon<r,$ 
$$
-\sd_{E_0}=\d_{E_0}\ge c_\Phi \d_{E_0}^{\Phi^o}  \ge  c_\Phi \dist_{\Phi^o}(\p W_{r}^\Phi,\p W_{\rho+\epsilon}^\Phi) = c_\Phi(r-\rho-\epsilon)
$$
in $W_{\rho+\epsilon}^\Phi\setminus F_{\tau*}$ and therefore,
$$
-I_2 \ge \tfrac{c_\Phi(r-\rho - \epsilon)}{\tau}\,|W_{\rho+\epsilon}^\Phi\setminus F_{\tau*}|.
$$
Finally, for a.e. $\epsilon>0$ with $\cH^{n-1}(\p W_{\rho+\epsilon}^\Phi\cap \p^* F_{\tau*})=0$ using the isoperimetric inequality \eqref{isoperimetric_ineq} we get
\begin{align*}
I_1 = & P_\Phi(W_{\rho+\epsilon}^\Phi) - P_\Phi(W_{\rho+\epsilon}^\Phi\cap F_{\tau*}) \le c_{\Phi,n} \Big(|W_{\rho+\epsilon}^\Phi|^{\frac{n-1}{n}} - |W_{\rho+\epsilon}^\Phi\cap F_{\tau*}|^{\frac{n-1}{n}}\Big) \\
= & c_{\Phi,n} |W_{\rho+\epsilon}^\Phi|^{\frac{n-1}{n}}\Big(1 - \Big|1 -\tfrac{|W_{\rho+\epsilon}^\Phi\setminus F_{\tau*}|}{|W_{\rho+\epsilon}^\Phi|}  \Big|^{\frac{n-1}{n}}\Big) < \tfrac{c_{\Phi,n}|W_{\rho+\epsilon}^\Phi\setminus F_{\tau*}|}{|W_{\rho+\epsilon}^\Phi|^{1/n}} = \tfrac{c_{\Phi,n}}{|W^\Phi|^{1/n}(\rho+\epsilon)}\,|W_{\rho+\epsilon}^\Phi\setminus F_{\tau*}|,
\end{align*}
where in the last inequality we used
\begin{equation}\label{obvious_estimate098}
(1-x)^\alpha > 1-x,\quad \alpha,x\in(0,1).
\end{equation}
Now using $-I_2 \le I_1+|I_3|$ and the above estimates for $I_i$ we get 
$$
\tfrac{c_\Phi(r-\rho - \epsilon)}{\tau} \le \tfrac{c_{\Phi,n}}{|W^\Phi|^{1/n}(\rho+\epsilon)} + \|f\|_\infty.
$$
Now letting $\epsilon\to0^+,$ and recalling $\rho\ge 4r/5$ we get \eqref{muhojirlar_til}.

Now let $\{E(\tau,k)\}$ be any flat flow starting from $E_0$ and associated to $\sF_{\beta,f},$ and let $\{F(\tau,k)_*\}$ be the flat flow starting from $F_0:=W_{R_0}^\Phi$ and associated to $\sF_{\beta_0,-f},$ consisting of the minimal minimizers. By the choice of $\beta_0,$ one has $\beta+\beta_0>0$ $\cH^{n-1}$-a.e. on $\p\Omega,$ and therefore, by Lemma \ref{lem:compare_setg} (b) $F(\tau,k)_*\cap E(\tau,k)=\emptyset.$
Let $k\mapsto \rho(\tau,k)$ be a nonincreasing sequence such that either $\rho(\tau,k) = \rho(\tau,k-1)$ or $W_{\rho(\tau,k)}^\Phi$ is the maximal Wulff shape contained  in $F(\tau,k)_*.$ By definition, $\rho(\tau,0)=R_0.$ 
By Claim 1, for any $k\ge1,$ we may assume 
\begin{equation}\label{anvar_tursun}
\rho(\tau,k-1) > \rho(\tau,k)\ge \rho(\tau,k-1) - \Big(C_8 + \tfrac{C_9}{\rho(\tau,k-1)}\Big)\tau,\quad \tau \le \tfrac{c_\Phi^2\theta^2\rho(\tau,k-1)}{25}.
\end{equation}
Let $k_0\ge1$ be some element for which $\rho(\tau,k_0)\ge R_0/2.$ By \eqref{anvar_tursun} for any $1\le k\le k_0$
$$
\rho(\tau,k) \ge \rho(\tau,0) - \sum_{i=0}^{k-1} \Big(C_8 + \tfrac{C_9}{\rho(\tau,i)}\Big)\tau \ge R_0 - \Big(C_8 + \tfrac{2C_9}{R_0}\Big)k\tau.
$$
Thus, if $\tau<\frac{c_\Phi^2\theta^2R_0^2}{100}$ and $k\tau \le \frac{R_0^2}{2C_8R_0 + 4C_9},$
then $\rho(\tau,k)\ge R_0/2.$ In particular, by the definition of $F_{\tau*}$
$$
W_{R_0/2}^\Phi(p) \cap E(\tau,k) = \emptyset,\quad 
0<\tau<\tfrac{c_\Phi^2\theta^2R_0^2}{100},\quad 0\le k\tau \le \tfrac{R_0^2}{2C_8R_0 + 4C_9}.
$$
\smallskip

{\it Step 2:} $W_{R_0}^\Phi(p)\setminus\cl{\Omega}\ne\emptyset,$ but $W_{\lambda_0R_0}^\Phi(p)\subset\Omega,$ where $\lambda_0:=\frac{\beta_0}{8\Phi(\be_n)}.$ 
\bigskip

By step 1 (applied with $R_0:=\lambda_0R_0$)
$$
W_{\lambda_0R_0/2}^\Phi(p)\cap E(\tau,k) = \emptyset, \quad 0<\tau<\tfrac{c_\Phi^2\theta^2\lambda_0^2R_0^2}{100},\quad 0\le k\tau \le \tfrac{\lambda_0^2R_0^2}{2C_8\lambda_0R_0 + 4C_9}.
$$
\smallskip

\begin{figure}[htp!]
\includegraphics[width=0.5\textwidth]{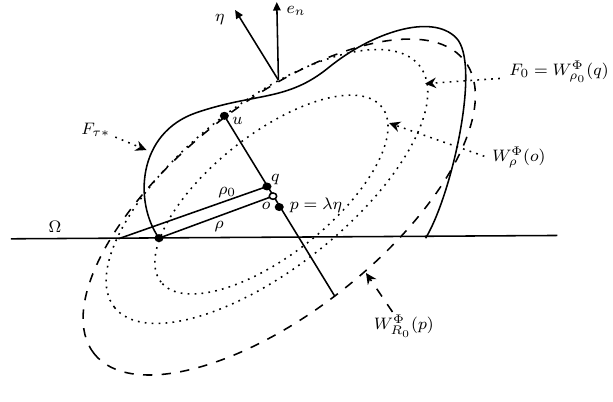}
\caption{\small Winterbottom shapes contained in $F_0$ and $F_{\tau*}$.}\label{fig:wulffs_evolving}
\end{figure}

{\it Step 3:} $W_{\lambda_0R_0}^\Phi(p)\setminus \cl{\Omega}\ne\emptyset$ i.e., $p\cdot\be_3<\lambda_0R_0.$ 
\bigskip

Fix any $\eta\in\p\Phi(\be_n),$ i.e. any vector in $\R^n$ satisfying $\Phi^o(\eta)=1$ and $\eta\cdot \be_n = \Phi(\be_n).$ When $\Phi$ is smooth, $\eta=\nabla\Phi(\be_n)$ and is an outer normal to the Wulff shape $W^\Phi$ at $\be_n^{\Phi^o}.$ For any $r>0$ let us define 
$$
W_r:=\Omega\cap W_{r}^\Phi\Big(\tfrac{\beta_0r\eta}{\Phi(\be_n)}\Big);
$$
Since 
$
\be_n \cdot \tfrac{\beta_0\eta}{\Phi(\be_n)} = \beta_0,
$
$W_r$ is the horizontal translation of the Winterbottom shape $\Omega\cap W_r^\Phi(\beta_0r\be_n)$ with contact angle $\beta_0,$ and thus, is itself 
a Winterbottom shape. 
For simplicity of the presentation, horizontally translating if necessary, we assume that $p = \lambda\eta$ for some $\lambda>0.$ One can readily check that $W_r\subset W_{r'}$ if $r<r'.$ 

By the assumption of step 3, there exists a Winterbottom shape $W_r\subset W_{R_0}^\Phi(p).$ In the notation of Figure \ref{fig:wulffs_evolving} 
let $W_{\rho_0}\subset W_{R_0}^\Phi(p)$ be the largest. Since $W_{\rho_0}$ is the translation in the $\eta$-direction of a Wulff shape $W_{\rho_0}^\Phi(p),$  we have 
\begin{equation}\label{tempeture_right}
R_0 = \rho_0 + \Phi^o(q-p)\quad \Longrightarrow \quad
\rho_0 = \frac{R_0 + \lambda}{1+\frac{\beta_0}{\Phi(\be_n)}}.
\end{equation}
Let $F_0:=W_{\rho_0}$ and $F_{\tau*}$ be the minimal minimizer of $\sF_{\beta_0,-f}(\cdot; F_0,\tau,k).$ As in step 1 assuming $0<\tau<\frac{\theta^2\rho_0^2}{25C_\Phi^2}$ and using \eqref{Linfty_estimos} we can show that 
\begin{equation}\label{nima_bor_orqasida}
\Omega\cap W_{4\rho_0/5}^\Phi(q) \subset  F_{\tau*}.
\end{equation}
Next, let $W_\rho$ and $W_{\bar \rho}$ be the largest Winterbottom shapes contained in $F_{\tau*}$ and in $ W_{4\rho_0/5}^\Phi(q),$ respectively. By \eqref{nima_bor_orqasida} and \eqref{tempeture_right} (applied with $R_0:=4\rho_0/5$)
\begin{equation}\label{headliefht}
\rho\ge \bar\rho = \frac{4\rho_0/5 + \lambda}{1+\frac{\beta_0}{\Phi(\be_n)}}\ge \frac{4\rho_0/5}{1+\frac{\beta_0}{\Phi(\be_n)}}.
\end{equation}

{\it Claim 2:} either $\rho\ge \rho_0$ or 
\begin{equation}\label{stoyte09i}
\rho_0>\rho\ge \rho_0 - \Big(C_{10}+\frac{C_{11}}{\rho_0}\Big)\tau,
\end{equation}
where $C_{10},C_{11}>0$ are some constants depending only on $\Phi,$ $\beta_0,$ $n$ and $\|f\|_\infty.$

Inded, assume that $\rho<\rho_0$ and fix $\epsilon\in(0,\rho_0-\rho).$ Consider the Winterbottom shape $W_{\rho+\epsilon}.$ The maximality of $\rho$ implies $|W_{\rho+\epsilon}\setminus F_{\tau*}|\searrow0$ as $\epsilon\to0^+.$  As in step 1, by the minimality of $F_{\tau*},$
\begin{align*}
0\le & \sF_{\beta_0,f}(F_{\tau*}\cup W_{\rho+\epsilon};F_0,\tau,k) - \sF_{\beta_0,f}(F_{\tau*};F_0,\tau,k) \\
= & \cC_{\beta_0}(F_{\tau*}\cup W_{\rho+\epsilon})-\cC_{\beta_0}(F_{\tau*}) - \frac{1}{\tau}\int_{W_{\rho+\epsilon}\setminus F_{\tau*}} \d_{F_0}dx + \frac{1}{\tau}\int_{k\tau}^{(k+1)\tau}ds\int_{W_{\rho+\epsilon}\setminus F_{\tau*}} f(s,x)dx\\
= & I_1+I_2+I_3.
\end{align*}
By the boundedness of $f,$
$$
|I_3| \le \|f\|_\infty|W_{\rho+\epsilon}\setminus F_{\tau*}|.
$$
Moreover, since the  centers of the Winterbottom shapes $W_{\rho_0}$ and $W_\rho$ lie on the same line, 
$$
-I_2 =
\d_{F_0} \ge c_\Phi\d_{F_0}^{\Phi^o} \ge c_\Phi(\rho_0-\rho-\epsilon-\Phi^o(q-o_\epsilon)) \quad\text{in $W_{\rho+\epsilon}\setminus F_{\tau*},$}
$$
where $o_\epsilon\in\Omega$ and $q$ are the centers of $W_{\rho+\epsilon}$ and $W_{\rho_0}.$ Since $o_\epsilon=\frac{\beta_0(\rho+\epsilon)\eta}{\Phi(\be_n)}$ and $q=\frac{\beta_0 \rho_0\eta}{\Phi(\be_n)},$ 
$$
\frac{1}{\tau}\int_{W_{\rho+\epsilon}\setminus F_{\tau*}} \d_{F_0}dx\ge \frac{c_\Phi(\rho_0-\rho-\epsilon)}{\tau}\,\Big(1-\tfrac{\beta_0}{\Phi(\be_n)}\Big)|W_{\rho+\epsilon}\setminus F_{\tau*}|.
$$
Finally, for a.e. $\epsilon$ with $\cH^{n-1}(\pOmega W_{\rho+\epsilon}\cap \pOmega F_{\tau*})=0$ using \eqref{winterbot_ineq9i8} and \eqref{obvious_estimate098} we get
\begin{align*}
I_1= & 
\cC_{\beta_0}(W_{\rho+\epsilon}) - \cC_{\beta_0}(W_{\rho+\epsilon}\cap F_{\tau*}) \le  c_{\Phi,\beta_0,n} \Big(|W_{\rho+\epsilon}|^{\frac{n-1}{n}} - |W_{\rho+\epsilon}\cap F_{\tau*}|^{\frac{n-1}{n}} \Big)\\
\le &  c_{\Phi,\beta_0,n} |W_{\rho+\epsilon}|^{\frac{n-1}{n}} \Big(1 - \Big|1- \tfrac{|W_{\rho+\epsilon}\setminus F_{\tau*}|} {|W_{\rho+\epsilon}|}\Big|^{\frac{n-1}{n}} \Big)
\le \tfrac{c_{\Phi,\beta_0,n}}{|W_{1}|^{1/n}(\rho+\epsilon)}\,|W_{\rho+\epsilon}\setminus F_{\tau*}|.
\end{align*} 
Using $-I_2\le I_1+|I_3|$ and letting $\epsilon\to0$ we get 
$$
\tfrac{c_\Phi(\rho_0-\rho)}{\tau}\Big(1-\tfrac{\beta_0}{\Phi(\be_n)}\Big) \le
\tfrac{c_{\Phi,\beta_0,n}}{|W_{1}|^{1/n}\rho} + \|f\|_\infty.
$$
Now in view of \eqref{headliefht} we deduce \eqref{stoyte09i} for suitable $C_{10},C_{11}>0$ depending only on $\Phi,n,\beta_0$ and $\|f\|_\infty$.

Now take any flat flows $\{E(\tau,k)\}$ starting from $E_0,$ and given $\rho_0$ in \eqref{tempeture_right}, let the numbers $\rho_0=\rho(\tau,0)\ge \rho(\tau,1)\ge \ldots$ be defined as follows. for each $k\ge1,$ if $\rho(\tau,k)<\rho(\tau,k-1),$ then $W_{\rho(\tau,k)}$ is the maximal Winterbottom shape staying inside the minimal minimizer of $\sF_{\beta_0,-f}(\cdot;W_{\rho(\tau,k-1)},\tau,k).$ By the choice of $R_0$ and the definition of $\rho_0,$
$$
E(\tau,0)\cap W_{\rho(\tau,0)} = \emptyset.
$$
Therefore, applying Lemma \ref{lem:compare_setg} (b) inductively, we deduce 
\begin{equation}\label{compare_geded}
E(\tau,k)\cap W_{\rho(\tau,k)} = \emptyset,\quad k=0,1,2,\ldots.
\end{equation}
As in step 1, let $k_0\ge1$ be such that $\rho(\tau,k_0)\ge \rho_0/2$ and assume that $\tau<\frac{\theta^2\rho_0^2}{100C_\Phi^2}.$ Then $\tau<\frac{\theta^2\rho(\tau,k-1)^2}{25C_\Phi^2}$ for any $1\le k\le k_0$ and hence, by Claim 2,
$$
\rho(\tau,k) \ge \rho(\tau,k-1) + \Big(C_{10} + \tfrac{C_{11}}{\rho(\tau,k-1)}\Big)\tau,\quad k=1,\ldots,k_0.
$$
From this inequality we deduce
$$
\rho(\tau,k) \ge \rho_0 -  \Big(C_{10} + \tfrac{2C_{11}}{\rho_0}\Big)k\tau.
$$
Thus, if we choose 
$0\le k\tau \le \frac{\rho_0^2}{2C_{10}\rho_0 + 4C_{11}},$ 
then $\rho(\tau,k)\ge \rho_0/2.$
Notice that by \eqref{compare_geded} for such $\tau$ and $k$ we have $E(\tau,k)\cap W_{\rho_0/2}=\emptyset.$

Let us show $W_{\lambda_0R_0}(p)\subset W_{\rho_0/2}.$ Since $p=\lambda\eta$ and the Wulff shape $W_{\rho_0/2}$ is centered at $\frac{\beta_0\rho_0 \eta}{2\Phi(\be_n)},$ it suffices to show 
$$
\lambda + \lambda_0R_0 \le \frac{\beta_0\rho_0 }{2\Phi(\be_n)}.
$$ 
By assumption of step 3 and the choice of $p$, the origin lies in $W_{\lambda_0R_0}^\Phi(p),$ and therefore, 
$$
\lambda = \Phi^o(p-0)\le \lambda_0R_0,
$$
and hence, by the choice of $\rho_0$ and assumption $\beta_0<\Phi(\be_n)$ we obtain
$$
\frac{\beta_0\rho_0 }{2\Phi(\be_n)} >\frac{\beta_0 R_0}{4\Phi(\be_n)} = 2\lambda_0R_0 \ge \lambda + \lambda_0R_0.
$$
Thus, $W_{\lambda_0R_0}^\Phi(p)\subset W_{\rho_0/2}.$

Theorem  is proved.
\end{proof}

Notice that when the forcing $f$ is zero, then the coefficients $C_9$ and $C_{11}$ in claim 1 and 2 can be taken $0.$

\section{Smooth $\Phi$-curvature flow of hypersurfaces with boundary}\label{sec:smooth_flow}

Throughout this section $\alpha\in(0,1]$ stands for a constant representing the H\"olderinanity exponent, $\Phi$ is an elliptic at least $C^{3+\alpha}$-anisotropy in $\R^n,$ $\beta\in C^{1+\alpha}(\p\Omega)$ satisfying \eqref{beta_condio} and $f\in C^{1+\frac{\alpha}{2},1+\alpha}(\R_0^+\times \cl{\Omega}).$  

In this section we prove that the evolution equation \eqref{aniso_mce} is well-posed and is solvable even in a more general setting.

\begin{definition}[\textbf{$\Phi$-curvature flow of hypersurfaces}]
A family $\{\Gamma(t)\}_{t\in[0, T)}$ (for some $T>0$) of smooth hypersurfaces in $\Omega$ with boundary is called a (smooth) \emph{$\Phi$-curvature flow, starting from a smooth hypersurface $\Gamma_0\subset \Omega$,
with forcing $f$ and anisotropic contact angle $\beta$} provided that
\begin{equation}\label{main_mean_curvature_PDE}
\begin{cases}
v_{\Gamma(t)} = -\kappa_{\Gamma(t)}^\Phi -f(t,\cdot) & \text{on $\Omega\cap \Gamma(t),$  $t\in [0,T),$}\\[1mm]
\p\Gamma(t)\subset \p\Omega & \text{$t\in [0,T),$} \\[1mm]
\nabla\Phi(\nu_{\Gamma(t)}) \cdot \be_n = -\beta & \text{on $\p\Gamma(t),$ $t\in [0,T),$} \\[1mm]
\Gamma(0) = \Gamma_0.
\end{cases}
\end{equation}
\end{definition}

\subsection{Solvability of \eqref{main_mean_curvature_PDE}}

In this section following the ideas of \cite{BKh:2018,KKR:1995} we prove the following short-time existence of the $\Phi$-curvature flow. 

\begin{theorem}\label{teo:short_time_flow}
Let $\Gamma_0\subset\Omega$ be a bounded $C^{2+\alpha}$-hypersurface with boundary on $\p\Omega$ oriented by a unit normal field $\nu_{\Gamma_0}$ and satisfying the anisotropic contact angle condition
$$
\nabla\Phi(\nu_{\Gamma_0})\cdot \be_n = -\beta\quad\text{on $\p\Gamma_0.$}
$$
Then there exist $T>0,$ depending only on $\Phi,$ $H,$ $\beta,$ $\|II_{\Gamma_0}\|_\infty$ and a $C^{1+\frac{\alpha}{2}}$-in time family $\{\Gamma(t)\}_{t\in[0,T]}$ of $C^{2+\alpha}$-hypersurfaces in $\Omega$ which satisfies \eqref{main_mean_curvature_PDE}.
\end{theorem}

To prove this theorem we first translate geometric PDE \eqref{main_mean_curvature_PDE} into a nonlinear parabolic system using parametrizations. For the convenience of the reader we  divide the proof of the theorem into smaller steps.

\subsubsection{H\"older spaces}

For $T>0,$ an open set $\sU\subset\R^{n-1},$ and a noninteger real number $\gamma>0$ let  $C_T^{\gamma/2,\gamma}$ be the Banach space $C^{\gamma/2,\gamma}([0,T]\times \cl{\sU})$ of the H\"older continuous functions for which
$$
\|w\|_{\gamma,T}:= \sum_{0\le i\le [\gamma/2]} \Big\|\tfrac{\p^iw}{\p t^i}\Big\|_\infty + \sum_{0\le |\mu|\le [\gamma]} \Big\|\tfrac{\p^\mu w}{\p x^\mu} \Big\|_\infty
+ \Big[\tfrac{\p^{[\gamma/2]} w}{\p t^{[\gamma/2]}}\Big]_{t,\gamma/2-[\gamma/2]} + \sum_{|\mu| = [\gamma]} \Big[\tfrac{\p^{[\gamma]} w}{\p x^{[\gamma]}}\Big]_{x,\gamma-[\gamma]} 
$$
is finite.
Here $[x]$ is the integer part of $x\in\R,$ $\mu=(\mu_1,\ldots,\mu_{n-1})\in\N_0^{n-1}$ is a multiindex and $|\mu|:=\mu_1+\ldots+\mu_{n-1}$,
$$
\tfrac{\p^\mu}{\p x^\mu} = \tfrac{\p^{|\mu|}}{\p^{\mu_1} x_1\ldots \p^{\mu_{n-1}} x_{n-1}},
$$
for a continuous function $f\in C^0([0,T]\times \cl{\sU})$ 
$$
\|f\|_\infty = \max_{[0,T]\times \cl{\sU}}\,|f|,
$$
and for $\theta\in (0,1]$ and $f\in C^0([0,T]\times \cl{\sU}),$
$$
[f]_{t,\theta} = \sup_{(t,x),(s,x)\in [0,T]\times \cl{U},s \ne t}\,\,\, \tfrac{|f(t,x) - f(s,x)|}{|t-s|^{\theta}},
\qquad 
[f]_{x,\theta} = \sup_{(t,x),(t,y)\in [0,T]\times \cl{U},x \ne y}\,\,\, \tfrac{|f(t,x) - f(t,y)|}{|x-y|^{\theta}}.
$$
We consider H\"older spaces $C_T^{\gamma/2,\gamma}$ only for $\gamma=\alpha$ and $\gamma=2+\alpha$ for some $\alpha\in(0,1].$ By $[ C_T^{\gamma/2,\gamma}]^m$ we denote the Banach space of vectors $f=(f_1,\ldots,f_m)$ where each $f_i\in C_T^{\gamma/2,\gamma},$ and the norm of $f$ is given by
$$
\|f\|_{\gamma,T} = \sum\limits_{i=1}^m \|f_i\|_{\gamma,T}.
$$

\subsubsection{Introducing the parametrization} 
For simplicity we assume that $\Gamma(t)$ are parametrized by a single chart $p:[0,T]\times \sU \to \Omega,$ where $\sU$ is a bounded $C^{2+\alpha}$-open set in $\R^{n-1}.$ 
In this case, we write
$$
p = 
\begin{bmatrix}
p^1\\
\ldots\\
p^n
\end{bmatrix}
\quad\text{and}\quad 
p_x = 
\begin{bmatrix}
p_x^1\\
\ldots\\
p_x^n
\end{bmatrix}
=
\begin{bmatrix}
p_{x_1}^1 & \ldots & p_{x_{n-1}}^1 \\
& \ldots & \\
p_{x_1}^n & \ldots & p_{x_{n-1}}^n,
\end{bmatrix}
$$
and recall that $\{p_{x_i}\}_{i=1}^{n-1}$ is the set of basis vectors of the tangent hyperplane of $\Gamma(t),$ 
\begin{equation}\label{danf76_normal}
\nu_\Gamma(p_x): = \frac{N(p_x)}{|N(p_x)|},\qquad 
N(p_x) = p_{x_1}\times \ldots \times p_{x_{n-1}} 
=
\det \,
\begin{bmatrix}
\be_1 & \ldots & \be_n\\
p_{x_1}^1 & \ldots & p_{x_1}^n\\
& \ldots & \\
p_{x_{n-1}}^1 & \ldots & p_{x_{n-1}}^n 
\end{bmatrix}
,
\end{equation}
is its  ``outward'' unit normal field, where $\times$ is the vector product of two vectors,
$$
g_{ij} = p_{x_i}\cdot p_{x_j},\quad p_{x_i} = \tfrac{\p p}{\p x_i},\quad i,j=1,\ldots,n-1,
$$
are entries of the first fundamental form of $\Gamma(t),$ $\{g^{ij}\}$ is its inverse and 
$$
h_{ij} = - \nu_{\Gamma(t)}\cdot p_{x_ix_j}, \quad p_{x_ix_j} = \tfrac{\p^2p}{\p x_i\p x_j}, \quad i,j=1,\ldots,n-1,
$$
are the entries of the second fundamental form of $\Gamma(t)$. Under these notations, the $\Phi$-curvature of $\Gamma(t) = p(\{t\}\times \sU)$ is represented as (see e.g. \cite{CvM:2004})
$$
\kappa_{\Gamma(t)}^\Phi = \sum\limits_{i,j=1}^{n-1} g^{ij} \, [h_{\Gamma(t)}]_{ij},
$$
where 
$$
[h_{\Gamma(t)}]_{ij}:= \Big(\nabla^2\Phi^o(\nu_{\Gamma(t)})\, \tfrac{\p \nu_{\Gamma(t)}}{\p x_i}\Big)\cdot p_{x_j},\quad i,j=1,\ldots,n-1,
$$
is the entries of the anisotropic version of the second fundamental form. Here, $\frac{\p \nu_{\Gamma(t)}}{\p x_i}$ is understood as a covariant derivative of $\nu_{\Gamma(t)}$ in $\R^n$ and can be defined as 
$$
\frac{\p \nu_{\Gamma(t)}}{\p x_i}=
\sum\limits_{k,l=1}^{n-1} h_{ik}\, g^{kl}\, p_{x_l} = - \sum\limits_{k,l=1}^{n-1} \left(\nu_{\Gamma(t)}\cdot  p_{x_ix_k}\right) \,g^{kl} \,p_{x_l}.
$$
Then the normal velocity of $\Gamma(t)$ is defined as 
$$
v_{\Gamma(t)} = - p_t\cdot \nu_{\Gamma(t)}
$$
and the $\Phi$-curvature is defined as 
$$
\kappa_{\Gamma(t)}^\Phi =  \sum\limits_{i,j,k,l=1}^{n-1} g^{ij}g^{kl} \Big( [\nabla^2\Phi(\nu_{\Gamma(t)})p_{x_l} ]\cdot p_{x_j}\Big)\,p_{x_ix_k}\cdot \nu_{\Gamma(t)}.
$$
This suggests to choose the tangential velocity such that the equation $v_{\Gamma} = -\kappa_{\Gamma}^\Phi-f$ is represented by means of the parametrization $p$ as 
$$
p_t  = \sum\limits_{i,j,k,l=1}^{n-1} g^{ij}g^{kl} \Big([\nabla^2\Phi(\nu_{\Gamma(t)})p_{x_l}]\cdot p_{x_j}\Big)\,p_{x_ix_k} + f\nu_{\Gamma(t)}\quad\text{on $\p\Gamma(t).$}
$$
The boundary condition $\p\Gamma(t)\subset\p\Omega$ is equivalent to $p\cdot \be_n = 0$ on $\p \sU$ and since $\nabla\Phi$ is positively $0$-homogeneous, the anisotropic contact angle condition $\nabla\Phi(\nu_{\Gamma(t)})\cdot \be_n = -\beta$ on $\p\Gamma(t)$ together with \eqref{danf76_normal} becomes as 
$\nabla\Phi(N(p_x))\cdot \be_n = -\beta(p)$ on $\p \sU.$

As in \cite{BKh:2018,KKR:1995}, to keep the presentation simpler, we assume that $\Gamma_0$ admits a parametrization $p^0:\sU\to\Omega$ with  the property 
\begin{equation}\label{initially_good}
\begin{cases}
p^0(x)\cdot \be_n = 0,\\ 
\nabla\Phi(N(p_x^0(x)))\cdot \be_n =-\beta(p^0(x)),\\  
\nabla [\nabla_n \Phi( N(p_x^0(x)) )]
= \mu_0(x)\sum\limits_{i=1}^{n-1} \cn_i^0(x) p_{x_i}^0(x)
\end{cases}
\end{equation}
for $x=(x_1,\ldots,x_{n-1})\in \p \sU,$
where $\nabla_n\Phi = \nabla\Phi\cdot\be_n,$ $\cn^0:=(\cn_1^0,\ldots,\cn_{n-1}^0)$ is the outer unit normal to $\p \sU$ and $\mu_0(x)$ is a scaling factor. The first condition in \eqref{initially_good} maintains that $\p\Gamma_0\subset\p\Omega,$ while the second one is the anisotropic contact angle condition. These two conditions are related solely to the geometry of $\Gamma_0.$ The third condition in \eqref{initially_good} is possible since 
$
\nabla^2\Phi(N)N = \nabla[\nabla\Phi(N)]\cdot N = 0,
$
which in particular implies $\nabla [\nabla_n \Phi( N(p_x^0(x)) )]\cdot N = 0.$ In view of the ellipticity, $\nabla_{nn}^2\Phi(N)>0$ and hence, $\mu_0>0.$  

To make the problem well-posed we still need to impose $n-2$ conditions on the boundary which should determine the boundary tangential velocity of $\Gamma(t).$ Let $\tau_1^0,\ldots,\tau_{n-2}^0$ be the basis of the tangent plane of $\p\Omega\cap \p \Gamma_0.$ We assume that 
\begin{equation*}
\sum\limits_{i=1}^{n-1}\cn_i^0 \,p_{x_i} \cdot \tau_j^0(p^0) = \sum\limits_{i=1}^{n-1} \cn_i^0 \, p_{x_i}^0 \cdot \tau_j^0(p^0),\quad j=1,\ldots,n-2,
\end{equation*}
for $x=(x_1,\ldots,x_{n-1})\in\p \sU.$

Now \eqref{main_mean_curvature_PDE} is represented as
\begin{equation}\label{main_mcf_equation}
\begin{cases}
p_t  = \sum\limits_{i,j,k,l=1}^{n-1} g^{ij}(p_x)g^{kl}(p_x) \Big([\nabla^2\Phi(\nu(p_x))p_{x_l}]\cdot p_{x_j}\Big)\,p_{x_ix_k} + f(t,p)\nu(p_x) \quad \text{in $[0,T]\times \sU,$}\\[1mm]
p\cdot \be_n = 0 \qquad \text{on $[0,T]\times \p \sU,$}\\[1mm] 
\nabla\Phi(N(p_x)) \cdot \be_n = -\beta(p) \qquad  \text{on $[0,T]\times \p\sU,$}\\[1mm]
\sum\limits_{j=1}^{n-1} \cn_j^0 \,p_{x_j}\cdot \tau_i^0(p^0) = \sum\limits_{j=1}^{n-1} \cn_j^0 \, p_{x_j}^0 \cdot \tau_i^0(p^0) \qquad \text{in $[0,T]\times \p\sU$, $i=1,\ldots,n-2,$}\\[1mm]
p(0,\cdot) = p^0,
\end{cases}
\end{equation}
where  $\nu:=\nu_\Gamma$ is given as in \eqref{danf76_normal}.

Now we linearize this system around initial and boundary conditions, solve the linearized problem using Solonnikov theory \cite{Solonnikov:1965}, and then apply a fixed point theory in H\"older spaces to show the solvability of \eqref{main_mcf_equation}.

\subsubsection{Linearization of the system \eqref{main_mcf_equation}}

Using \eqref{initially_good} we rewrite \eqref{main_mcf_equation} as 
\begin{equation}\label{iting_sotmay_piting_sot}
\Big[\cA w, \cP w, \cC w, \cT w, \cI w\Big] = [\bar f,0,\bar b,0,p^0] + \Big[F(w,p^0),0, B(w,p^0),0,0\Big],
\end{equation}
where $w\in [C_{T}^{1+\frac{\alpha}{2},2+\alpha}]^n$ for some $T>0,$ which will be chosen later,
\begin{align*}
& \cA w:=w_t  - \sum\limits_{i,j,k,l=1}^{n-1} g^{ij}(p_x^0)g^{kl}(p_x^0) \Big([\nabla^2\Phi(\nu(p_x^0))p_{x_l}^0]\cdot p_{x_j}^0\Big)\,w_{x_ix_k}, \\
& \cP w = w\cdot \be_n, \\
& \cC w = \nabla^2\Phi(N(p_x^0))\, \nabla N(p_x^0)[w_x] \cdot \be_n + \nabla\beta(p^0)\cdot w, \\
& \cT w = \sum\limits_{j=1}^{n-1}\cn_j^0 w_{x_j}\cdot \tau_i^0(p^0),\\
& \cI w = w(0,\cdot)
\end{align*}  
are homogeneous linear operators, a linearized part of the system \eqref{main_mcf_equation}, where 
$$
\nabla N(p_x^0)[w_x] = \Big(\sum_{i=1}^n\sum_{j=1}^{n-1} \nabla_{p_{x_j}^i} N^1(p_x^0)\,w_{x_j}^i,\ldots,\sum_{i=1}^n\sum_{j=1}^{n-1} \nabla_{p_{x_j}^i} N^n(p_x^0)\,w_{x_j}^i\Big)^T,
$$
the vector-functions
\begin{align*}
& \bar f = f(\cdot,p^0)\nu(p_x^0),\\
& \bar b = \nabla^2\Phi (N(p_x^0))\, \nabla N(p_x^0)[p_x^0] \cdot \be_n + \nabla\beta(p^0)\cdot p^0,
\end{align*}
are the main parts of the right-hand side (independent of $w$) after linearization, and 
\begin{align*}
& 
\begin{aligned}
F(w,p^0) = & \sum\limits_{i,j,k,l=1}^{n-1} g^{ij}(w_x)g^{kl}(w_x) \Big([\nabla^2\Phi (\nu(w_x))w_{x_l}]\cdot w_{x_j}\Big)\,w_{x_ix_k} + f(t,w)\nu(w_x)\\
& -  \sum\limits_{i,j,k,l=1}^{n-1} g^{ij}(p_x^0)g^{kl}(p_x^0) \Big([\nabla^2\Phi (\nu(p_x^0))p_{x_l}^0]\cdot p_{x_j}^0\Big)\,w_{x_ix_k} - f(t,p^0)\nu(p_x^0),
\end{aligned}
\\
& 
\begin{aligned}
B(w,p^0) = & -\Big[\nabla\Phi(N(w_x)) - \nabla\Phi (N(p_x^0)) - \nabla^2\Phi(N(p_x^0))\,\nabla N(p_x^0)[w_x-p_x^0]\Big]\cdot \be_n\\
&
- \Big[\beta(w) - \beta(p_x^0) - \nabla\beta(p_x^0)[w-p^0]\Big]\,|N(p_x^0)|
\end{aligned}
\end{align*}
are nonlinear parts. 
Notice that $F(p^0,p^0) = 0$ and $B(p^0,p^0) = 0.$ 

\subsubsection{Parabolicity of \eqref{iting_sotmay_piting_sot}}

Let us show that the linear operator $\cA$ in the system \eqref{iting_sotmay_piting_sot} is parabolic in the sense of Solonnikov \cite[pp. 9]{Solonnikov:1965}, the linear operators $[\cP,\cC,\cT]$ satisfy the complementary conditions at the boundary $[0,T]\times \p \sU$ and at the initial time $t=0$ \cite[pp. 11-12]{Solonnikov:1965}, and the boundary conditions and initial datum in the right-hand side of \eqref{iting_sotmay_piting_sot} are compatible of order $0$ \cite[pp. 87]{Solonnikov:1965}.
\smallskip

{\it Parabolicity of $\cA.$} For $(t,x)\in[0,T]\times \cl{\sU},$ $z\in\C$ and $\xi\in\C^{n-1}$ let $\sA(t,x,z,\xi)$ be the $n\times n$-diagonal matrix whose all diagonal entries are equal to
$$
p - \sum\limits_{i,j,k,l=1}^{n-1} g^{ij}(p_x^0)g^{kl}(p_x^0) \Big([\nabla^2\Phi (\nu(p_x^0))p_{x_l}^0]\cdot p_{x_j}^0\Big)\xi_i\xi_k,
$$
and $L(t,x,z,\xi):=\det\,\sA(t,x,z,\xi).$ Then for any $\xi\in \R^{n-1}$ the equation $L(t,x,z,i\xi) = 0$ in $z\in\C$ has a unique solution (with multiplicity $n$)
$$
z = -\sum\limits_{i,j,k,l=1}^{n-1} g^{ij}(p_x^0)g^{kl}(p_x^0) \Big([\nabla^2\Phi (\nu(p_x^0))p_{x_l}^0]\cdot p_{x_j}^0\Big)\xi_i\xi_k.
$$
Being a basis of the tangent hyperplane at $p^0(\cdot),$ $p_{x_i}^0(\cdot)$ are orthogonal to $\nu(p_x^0(\cdot)),$ and hence, using the ellipticity of $\Phi$ and Proposition \ref{prop:elliptic_anis_propo} (b) we find
\begin{align*}
z = & -\Big(\nabla^2\Phi (\nu(p_x^0))\Big[\sum\limits_{k,l}g^{kl}\xi_k p_{x_l}^0\Big], \,\Big[\sum_{ij}g^{ij}\xi_ip_{x_j}^0\Big]\Big) \le -\gamma \Big|\sum_{ij}g^{ij}\xi_ip_{x_j}^0\Big|^2 
\end{align*}
for some $\gamma:=\gamma(\Phi,n)>0.$
Since $p_{x_i}^0\cdot p_{x_j}^0 = g_{ij},$ $\{g^{ij}\}$ is the inverse matrix to $\{g_{ij}\}$ and $\{g^{ij}\}$ is positive definite (by the linear independence of $\{p_{x_i}^0\}$), 
$$
\Big|\sum_{ij}g^{ij}\xi_ip_{x_j}^0\Big|^2 = \sum_{i,j,k,l} g^{ij}g^{kl}\xi_i\xi_k \big(p_{x_j}^0 \cdot p_{x_l}^0\big) = \sum_{i,j,k,l} g^{ij}g^{kl}g_{jl} \xi_i\xi_k= \sum_{k,l} g^{kl} \xi_k\xi_l\ge \bar\gamma |\xi|^2
$$
for some $\bar\gamma>0$ depending only on $\Gamma_0.$ Thus, 
$
z \le -\gamma\bar\gamma|\xi|^2
$
and $\cA$ is (uniformly) parabolic.
\smallskip 

{\it Complementary condition for the boundary conditions}. 
Let $\sB_0(t,x,z,\xi)$ be the matrix, corresponding to the highest order part of the boundary operator $[\cP,\cB,\cT]$ whose entries are
$$
B_{kl}(t,x,z,\xi) = 
\begin{cases}
\delta_{ln} & k=1,\\[2mm]
\sum\limits_{i=1}^n\nabla_{ni}^2\Phi(N(p_x^0))\sum\limits_{j=1}^{n-1}\nabla_{p_{x_j}^l}N^i(p_x^0)\xi_j & k=2,\\[3mm]
\tau_{k-2}^{0,l}\,\cn\cdot \xi, & i=3,\ldots,n,
\end{cases}
$$
where $\delta_{xy}=1$ for $x=y$ and $=0$ for $x\ne y,$ and $l=1,\ldots,n.$ By the definition \eqref{danf76_normal} of $N$ and the third relation in \eqref{initially_good}
\begin{align*}
\sum\limits_{i=1}^n \nabla_{ni}^2\Phi(N(p_x^0))\nabla_{p_{x_j}^l}N^i(p_x^0) 
=
\det
\begin{bmatrix}
\nabla [\nabla_n\Phi(N(p_x^0))] \\
 p_{x_1}^0  \\
\ldots \\
p_{x_{j-1}^0}  \\
\be_l \\
p_{x_{j+1}^0} \\
\ldots \\
p_{x_{n-1}^0}
\end{bmatrix}
=
- \mu_0 \cn_j^0 
\det
\begin{bmatrix}
\be_l\\
p_{x_1}^0  \\
\ldots \\
p_{x_{n-1}^0}
\end{bmatrix}
= 
- \mu_0 \cn_j^0 N^l(p_x^0).
\end{align*}
Therefore, we have also 
$$
\sB_0(t,x,z,\xi) = 
\begin{bmatrix}
\be_n\\
-\mu_0 [\cn^0\cdot\xi]\,N(p_x^0) \\
[\cn^0\cdot\xi]\,\tau_1^0\\
\ldots \\
[\cn^0\cdot\xi]\,\tau_{n-2}^0
\end{bmatrix}
.
$$
By \cite[pp. 11]{Solonnikov:1965}, the complementary conditions at the boundary holds iff at every $(t,x)\in [0,T]\times \p\sU$ and every tangent vector $\zeta(x)\in\R^{n-1}$ of $\p \sU$ at $x,$ the rows of the matrix 
$$
\sD(t,x,z,i(\zeta + \lambda \cn)):=\sB_0(t,x,z,i(\zeta + \lambda \cn))\hat\sA(t,x,z,i(\zeta + \lambda \cn))
$$
are linearly independent modulo the polynomial
$$
M^+(t,x,z,\zeta; \lambda) := \Big(\lambda - \lambda_s^+(t,x,z,\zeta)\Big)^n,
$$
where $\hat \sA(t,x,z,\xi) := L(t,x,z,\xi)\sA(t,x,z,\xi)^{-1},$ $\lambda_s^\pm(t,x,z,\zeta)$ are the zeros (of multiplicity $n$) of $L(t,x,z,i(\zeta + \lambda\cn)) = 0$ in $\lambda$ for $\Re(z) \ge -\delta_1|\zeta|^2$ and $|z|^2 + |\zeta|^2>0$ with $\delta_1>0.$ In our case, $\hat\sA$ is the identity matrix multiplied by $(\lambda-\lambda_s^+)^{n-1}(\lambda-\lambda_s^-)^{n-1},$ and hence, in view of the explicit expression of $\sB_0,$ the compatibility condition is equivalent to the linear independence of the vectors 
\begin{equation}\label{nukkadisvhg67}
\be_n,\quad N(p_x^0),\quad \tau_1^0,\quad \ldots,\quad  \tau_{n-2}^0.
\end{equation}
Take $c_1,\ldots,c_n\in\R$ such that 
$$
c_1\be_n + c_2 N(p_x^0) + \sum_{i=3}^n c_i \tau_{i-2}^0 = 0.
$$
By definition, $\tau_j^0\cdot \be_n = 0$ and $\tau_j^0\cdot N(p_x^0) =0,$ and hence, from the linear independence of $\tau_j^0$ (being a basis) $c_i = 0$ for $i\ge3.$ Moreover, if $c_1\ne0$ (hence, $c_2\ne0$), then $\be_n = -\frac{c_2}{c_1}N(p_x^0),$ and therefore, by the angle-condition  (the second equality in \eqref{initially_good}) and the evenness of $\Phi$
$$
-\beta = \nabla\Phi(N(p_x^0))\cdot \be_n = \frac{-c_2}{c_1} \Phi(N(p_x^0)) = \frac{\sign c_2}{\sign c_1}\Phi \Big(\frac{-c_2}{c_1}N(p_x^0)\Big) = \frac{\sign c_2}{\sign c_1}\Phi(\be_n). 
$$
However, in view of \eqref{beta_condio} this equality cannot happen, and therefore, $c_1=c_2=0,$ i.e., the vectors in \eqref{nukkadisvhg67} are linearly independent.
\smallskip

{\it Complementary conditions for the initial datum.} Let $\sC$ be the identity matrix, which corresponds to the operator $\cC.$
By \cite[pp. 12]{Solonnikov:1965} the complementary condition for the initial datum is read as follows: for each $x\in \sU$ the rows of the matrix 
$$
\tilde \sD(x,z): = \sC(x,0,z) \hat\sA(0,x,z,0)
$$
are linearly independent modulo polynomial $z^n.$ As we have seen above $\hat\sA(0,x,z,0)$ is identity matrix multiplied by $z^{n-1},$ and hence, by the definition of $\sC,$ so is $\tilde \sD(x,z).$ Then clearly the rows of $\tilde \sD(x,z)$ are linearly independent modulo $z^n.$
\smallskip

{\it Compatibility conditions.} Notice that while linearizing we obtained the identity
$$
[\cP p^0,\cC p^0,\cT p^0]  = [0,b+B(p^0,p^0),0],
$$
which reads as the $0$-order compatibility of the boundary datum in the right-hand side of \eqref{iting_sotmay_piting_sot} with the initial datum in the sense of \cite[pp. 87]{Solonnikov:1965}.
\smallskip

\subsubsection{Solvability of \eqref{iting_sotmay_piting_sot}}
For $L>1$ and $T>0,$ let $X_{L,T}$ be the collection of all $w\in C_T^{1+\frac{\alpha}{2},2+\alpha}$ such that 
\begin{itemize}
\item[(1)] $w(0,\cdot) = p^0(\cdot)$ in $\cl{\sU},$

\item[(2)] $w\cdot \be_n = 0$ in $[0,T]\times \p \sU,$ %and $w\cdot \be_n > 0$ in $[0,T]\times \sU,$

\item[(3)] $\sum\limits_{i=1}^{n-1}n_i^0 w_{x_i}\cdot \tau_j^0(p^0) = 0$ in $[0,T]\times \p \sU,$

\item[(4)] the vectors $\{w_{x_i}\}_{i=1}^{n-1}$ are linearly independent,

\item[(5)] $\|w\|_{C^{1+\frac{\alpha}{2},1+\alpha}}\le L.$
\end{itemize}
Clearly, $X_{L,T}\ne\emptyset,$ since conditions (1), (2),(3) and (5) allow to construct $w$ first in the neighborhood of $\p \sU$, and then to extend to interior of $\sU$. The condition (4) holds at least for small $T;$ in fact, since 
$$
w_x(t,x) = p_x^0(x) + \int_0^tw_{tx}(s,x)ds, \quad w\in X_L(T),
$$
$\|p_x^0 - w_x(t,\cdot)\|_\infty\le LT$ whenever $t\le T.$ Thus, 
$$
\det(\{w_{x_i}\cdot w_{x_j}\}) = \det(\{p^0_{x_i}\cdot p^0_{x_j}\}) - C_1T,
$$
where $C_1>0$ depends only on $n,$ $\|p_x^0\|_\infty$ and $L.$ Thus, if we choose
$
T<T_1:=\frac{\det(\{p^0_{x_i}\cdot p^0_{x_j}\})}{C_1},
$
then $w_{x_i}$ are linearly independent. We can also show that $X_{L,T}$ is a closed convex subspace of $C_T^{1+\frac{\alpha}{2},2+\alpha}$.

Notice that for any $w\in X_{L,T}$ the vectors $[\bar f+F(w,p^0),0,\bar b+B(w,p^0),0,p^0]$ satisfy the $0$-order compatibility condition, and therefore, there exists a unique $\sS_w \in C_T^{1+\frac{\alpha}{2},2+\alpha}$ such that 
\begin{equation}\label{dandu836czb}
[\cA[\sS_w],\cP[\sS_w],\cC[\sS_w],\cT[\sS_w],\cI[\sS_w]] = [\bar f+F(w,p^0),0,\bar b+B(w,p^0),0,p^0] 
\end{equation}
and
\begin{equation}\label{ding_ding_ding_vottak}
\|\sS_w\|_{2+\alpha,T} \le C_0\Big(\|\bar f+F(w,p^0)\|_{\alpha,T} + \|\bar b+B(w,p^0)\|_{\alpha,T} + \|p^0\|_{\alpha,T}\Big) 
\end{equation}
for some $C_0>0$ (continuously) depending only on $\beta,$ $\Phi,$ $p^0,$ and also on $\sU$ and $n.$ By uniqueness and linearity, from \eqref{ding_ding_ding_vottak} for any $w_1,w_2\in X_{L,T}$ we have
$$
\|\sS_{w_1} - \sS_{w_2}\|_{2+\alpha,T} =\|\sS_{w_1-w_2}\|_{2+\alpha,T} \le C_0\Big(\|F(w_1,p^0) - F(w_2,p^0)\|_{\alpha,T} + \|B(w_1,p^0) - B(w_2,p^0)\|_{\alpha,T}\Big).
$$
Using the explicit expressions of $F$ and $B,$ the definition of $X_{L,T}$ and the equality
\begin{equation}\label{sadnua6}
u(t,x) = p^0(x) + \int_0^t u_t(s,x)ds,\quad w\in C^{1,0}([0,T]\times\cl{U}),
\end{equation}
we can compute
$$
\|F(w_1,p^0) - F(w_2,p^0)\|_{\alpha,T} \le C_1T \|w_1-w_2\|_{2+\alpha,T}
$$
and
$$
\|B(w_1,p^0) - B(w_2,p^0)\|_{\alpha,T} \le C_1T \|w_1-w_2\|_{2+\alpha,T}
$$
for some $C_1$ depending on $L$ but not on $T,w_1$ and $w_2.$  Thus, if we choose $T<T_2:=\frac1{C_0C_1},$ then $w\mapsto \sS_w$ is a contraction. To apply a fixed point theorem, it remains to show that $\sS_w\in X_{L,T}$ whenever $w\in X_{L,T}.$ The equalities (1)-(3) for $\sS_w$ follow from the system \eqref{dandu836czb}. Moreover, since $T<T_1,$ the vectors $\{(\sS_w)_{x_i}\}$ are also linearly independent. It remains to check condition (5). Consider the estimate \eqref{ding_ding_ding_vottak}. 
By definition of $F$ and $B$ (they are somehow estimated by a power of $L$ times the norm of $w-p^0$), 
$$
\|F(w,p^0)\|_{\alpha,T} \le C_2L^{10} T,\qquad \|B(w,p^0)\|_{\alpha,T} \le C_2L^{10} T,
$$
where $C_2$ does not depend on $T>0$ and $L>1,$
and hence, by \eqref{sadnua6} 
$$
\|\bar f+F(w,p^0)\|_{\alpha,T} + \|\bar b+B(w,p^0)\|_{\alpha,T} + \|p^0\|_{\alpha,T} \le \|\bar f\|_{\alpha,T} + \|\bar b\|_{\alpha,T} + \|p^0\|_{\alpha,T} + 2C_2L^{10} T.
$$
Now if we choose  
$$
L:= 1 + 2C_0[\|\bar f\|_{\alpha,T} + \|\bar b\|_{\alpha,T} + \|p^0\|_{\alpha,T} ],
$$
then $\sS_w\in X_{L,T}$
provided 
$$
T \le T_3:=\tfrac{1 + C_0[\|\bar f\|_{\alpha,T} + \|\bar b\|_{\alpha,T} + \|p^0\|_{\alpha,T} ]}{2C_0C_2L^{10}}.
$$
Now the Banach fixed point theorem implies that there exists a unique $w\in X_{L,T}$ which satisfies $\sS_w = w.$ Then \eqref{dandu836czb} implies that $w$ is a solution of \eqref{main_mcf_equation} for small $T>0.$ 

\subsection{Long-time evolution}

Applying Theorem \ref{teo:short_time_flow} inductively we obtain the following generalization of Theorem \ref{teo:short_intro}.

\begin{theorem}\label{teo:max_time_exist}
Let $\Gamma_0\subset\Omega$ be a bounded $C^{2+\alpha}$-hypersurface with boundary satisfying 
$$
\p\Gamma_0 \subset\p\Omega\quad \text{and} \quad \nabla\Phi( \nu_{\Gamma_0}) \cdot \be_n = -\beta \,\, \text{on $\p\Omega$.}
$$
Then there exists a maximal time $T^\dag>0$ and a smooth $\Phi$-curvature flow $\{\Gamma(t)\}_{t\in [0,T^\dag)}$ starting from $\Gamma_0,$ with the forcing $f$ and anisotropic contact angle $\beta.$
\end{theorem}

The term ``maximal'' refers to the fact that there is no smooth $\Phi$-curvature flow $\{\Gamma(t)\}_{t\in [0,T')}$ for any $T'>T^\dag.$ Notice that at the maximal time $T^\dag$ for the set $\Gamma(T^\dag)$ (defined for instance as a Kuratowski limit of $\Gamma(t)$ as $t\nearrow T^\dag$) at least one of the following holds (otherwise applying Theorem \ref{teo:short_time_flow} we would extend the flow slightly after $T^\dag$):
\begin{itemize}
\item $\Gamma(T^\dag)$ is not $C^2$-anymore (the curvature blows up),

\item $\cl{\Gamma(T^\dag)}$ is not injective,

\item some interior point of $\Gamma(T^\dag)$ touches to $\p\Omega$ (because of the forcing).
\end{itemize}
In this paper we do not deal with the singularity analysis.

\begin{remark}
The $\Phi$-curvature flow equation is represented by means of the signed distances as 
\begin{equation}\label{repos_sdist_curva}
\frac{\p}{\p t}\,\sd_{E(t)}(x) = \kappa_{E(t)}^\Phi(x) + f(t,x),\quad t\in [0,T^\dag),\,\, x\in \pOmega E(t).
\end{equation} 

\end{remark}

\subsection{Stability of the $\Phi$-curvature flow}

The classical mean curvature flow of boundaries has the following remarkable stability property: \it if $\{E(t)\}_{t\in [0,T^\dag)}$ is the smooth mean curvature flow of $C^{2+\alpha}$-sets, then for every $0<T<T^\dag$ there exists $\epsilon>0$ such that if $F(0)$ is such that $\p F(0)$ belongs to the $C^{2+\alpha}$-neighborhood of $\p E(0),$ then there exists a unique mean curvature flow $\{F(t)\}_{t\in [0,T')}$ starting from $F(0)$ and $T'> T$
\rm (see e.g. \cite[Theorem 7.1]{ATW:1993}).

In this section we prove that the flow solving \eqref{aniso_mce} admits such a stability property. As in \cite{Kholmatov:2023} we are mainly interested in droplets with non-empty contact  on $\p\Omega,$ and therefore, it is natural to restrict ourselves to the regular droplets without connected components ``hanging''  in $\Omega.$

\begin{definition}[\textbf{Admissibility}]\label{def:admissible_sets}
$\,$
\begin{itemize}
\item[\rm(a)] We say a bounded set $E\subset\Omega$  is \emph{admissible} provided there exist a bounded $C^{2+\alpha}$-open set $\sU\subset\R^{n-1}$ and a $C^{2+\alpha}$-diffeomorphism $p\in C^{2+\alpha}(\cl{\sU};\R^n)$ satisfying  
\begin{equation*} 
p[\sU]=\Gamma,\qquad 
p[\p\sU]=\p\Gamma, \qquad 
p\cdot \be_n>0 \,\,\, \text{in $\sU$}
\qquad\text{and}\qquad
p \cdot \be_n=0 \,\,\, \text{on $\p\sU,$}
\end{equation*}
where $\Gamma:=\pOmega E.$ Any such map $p$ is called a \emph{parametrization} of $\Gamma.$ 

\item[\rm(b)] We say $E$ is \emph{admissible with anisotropic contact angle $\beta$} if $E$ is admissible and
\begin{equation}\label{contact_condition09}
\nabla\Phi(\nu_E) \cdot \be_n = -\beta \quad\text{on $\p\Omega\cap \cl{\Gamma}$}.
\end{equation}
We call the number 
\begin{equation}\label{minheights}
h_E:=\min_{x\in\cl{\Gamma},\,\nu_{E}(x)=x+\be_n}\, \,\,x\cdot \be_n
\end{equation}
the \emph{minimal height} of $E$.
Since $E$ satisfies \eqref{contact_condition09} and $\beta$ satisfies  \eqref{beta_condio}, $h_E>0.$

\item[\rm(c)] Let $Q$ be a compact set in $\R^m$ for some $m\ge1.$ We say a family $\{E[q]\}_{q\in Q}$ of bounded subsets of $\Omega$ is \emph{admissible} if there exist $\alpha\in(0,1],$ a bounded $C^{2+\alpha}$-open set $\sU\subset\R^{n-1}$ and a map $p\in C^{2+\alpha,2+\alpha}(Q \times \cl{\sU};\R^n)$ such that $p[q,\cdot]$ is a parametrization of $\pOmega E[q].$  

\item[\rm(d)] We say a family $\{E[q,t]\}_{q\in Q,t\in [0,T)}$ of bounded subsets of $\Omega$ \emph{admissible} if for any $T'\in(0,T)$ there exist  $\alpha\in(0,1],$ a bounded $C^{2+\alpha}$-open set $\sU\subset\R^{n-1}$ and a map $p\in C^{2+\alpha,1+\frac{\alpha}{2},2+\alpha}(Q \times [0,T'] \times \cl{\sU};\R^n)$ such that $p[q,t,\cdot]$ is a parametrization of $\pOmega E[q,t].$  

\end{itemize}
\end{definition}

\begin{remark}\label{rem:admisible_props}
$\,$
\begin{itemize}
\item[\rm(a)] By definition, if $E$ is an admissible set, then the $C^{2+\alpha}$-surface $\Gamma:=\pOmega E$ is diffeomorphic to a bounded  smooth open set in $\R^{n-1}$ and not necessarily connected (clearly, boundaries of two connected components do not touch). In particular, $\Gamma$ cannot not have ``hanging'' components compactly contained in $\Omega.$ Moreover, its boundary $\p\Gamma$ lies on $\p\Omega$ and the relative interior of $\Gamma$ does not touch to $\p\Omega.$  

\item[\rm(b)] When $Q$ is empty in Definition \ref{def:admissible_sets} (d), then we simply write $\{E[t]\}_{t\in [0,T)}$ to denote the corresponding admissible family.

\item[\rm(c)] If $p\in C^{2+\alpha,1+\frac{\alpha}{2},2+\alpha}(Q \times [0,T'] \times \cl{\sU};\R^n)$ is a parametrization of $\{E[q,t]\}_{q\in Q,t\in [0,T)}$ and $\sU'\subset\R^{n-1}$ is a bounded $C^{2+\alpha}$ open set diffeomorphic to $\sU$ via a map $\psi:\sU'\to\sU,$ then $p[q,t,\psi(\cdot)]$ is also a parametrization of $E[q,t]$.
\end{itemize}
\end{remark}

Remark \ref{rem:admisible_props} (c) allows to introduce the closeness of the free boundaries of two droplets. 

\begin{definition}\label{def:def:distance_drops}
For any two admissible set $E_1$ and $E_2$ we write 
$$
\bar d(E_1,E_2) = \inf_{p_1,p_2} \|p_1-p_2\|_{C^{2+\alpha}(\cl{\sU})},
$$
where $p_i\in C^{2+\alpha}(\sU;\R^n)$ is a parametrization of $\pOmega E_i.$  Similarly if $\{E_1[q,t]\}_{q\in Q,t\in [0,T)}$ and $\{E_2[q,t]\}_{q\in Q,t\in [0,T)}$ are two admissible families, we write 
$$
\bar d(E_1(t),E_2(t)) = \inf_{p_1,p_2} \|p_1-p_2\|_{C^{2+\alpha,1+\frac{\alpha}{2},2+\alpha}(Q \times [0,T'] \times \cl{\sU})},
$$
where $p_i\in C^{2+\alpha,1+\frac{\alpha}{2},2+\alpha}(Q \times [0,T'] \times \cl{\sU};\R^n)$ is a parametrization of $E_i[\cdot,\cdot].$
\end{definition} 

One can readily check that the infimum in the definition of $\bar d$ is in fact a minimum.

As we have observed in the proof of Theorem \ref{teo:short_time_flow} the constants $C_0,C_1,C_2$ and bounds $T_1,T_2,T_3$ for local time $T$ continuously depend on $\|f\|_{C^{\alpha/2,\alpha}(\R_0^+\times\cl{\Omega})},$ $\|\beta\|_{C^{1+\alpha}(\p\Omega)},$ $\|\Phi\|_{C^{3+\alpha}(\S^{n-1})}$ and $\|p^0\|_{C^{2+\alpha}(\cl{\sU})}.$ 
This implies the following stability of the flow which generalizes \cite[Theorem 7.1]{ATW:1993}. 

\begin{theorem}[\textbf{Stability of $\Phi$-curvature flow}]\label{teo:stability_mcf}
Let $\Phi_0$ be an elliptic $C^{3+\alpha}$-anisotropy,  $\beta_0\in C^{1+\alpha}(\p\Omega)$ satisfy \eqref{beta_condio}, $f_0\in C^{\frac{\alpha}{2},\alpha}(\R_0^+\times \cl{\Omega}),$  and $\{E_0(t)\}_{t\in[0,T_0)}$ be a bounded smooth $\Phi_i$-curvature flow with forcing $f_i$ and anisotropic contact angle $\beta_i$ for some $T_0>0.$ Then for any $T\in (0,T_0)$ there exist $\epsilon_0>0$ and a nondecreasing function $\psi:\R_0^+\to\R_0^+$ with $\psi(0)=0$ with the following property.
For $i=1,2,$ let $\Phi_i$ be an elliptic $C^{3+\alpha}$-anisotropy, $\beta_i\in C^{1+\alpha}(\p\Omega)$ satisfying \eqref{beta_condio} and $f_i\in C^{\frac{\alpha}{2},\alpha}(\R_0^+\times \cl{\Omega})$ and $\Phi_i$-curvature flow $\{E_i(t)\}_{t\in t\in [0,T_i]}$ with forcing $f_i$ and anisotropic contact angle $\beta_i$ for some $T_i>0$ be such that 
$$
\|\Phi_i - \Phi_0\|_{C^{3+\alpha}(\cl{B_2(0)}\setminus  B_{1/2}(0))} + 
\|\beta_i - \beta_0\|_{C^{1+\alpha}(\p\Omega)} 
+
\|f_i-f_0\|_{C^{\alpha/2,\alpha}([0,T_0]\times\cl{\Omega})}
+\bar d(E_i(0,E_0(0)))
\le \epsilon_0.
$$
Then $T_i>T$ and 
\begin{equation}\label{smooth_dependence_initial}
\bar d(E_1(t),E_2(t)))\le \psi(\bar d(E_1(0),E_2(0)))),\quad t\in [0,T].
\end{equation}
\end{theorem}

In what follows we refer to \eqref{smooth_dependence_initial} as \emph{smooth dependence on the initial condition}.

Let us consider some applications of the stability.

\subsubsection{Comparison for $\Phi$-curvature flows}

The main result of this section is the following 

\begin{theorem}[\textbf{Strong comparison}]\label{teo:comparison_mcf}
Let $\Phi$ be an elliptic $C^{3+\alpha}$-anisotropy,  $\beta_i\in C^{1+\alpha}(\p\Omega)$ satisfy \eqref{beta_condio} and $f_i\in C^{\frac{\alpha}{2},\alpha}(\R_0^+\times \cl{\Omega}),$ $\{E_i(t)\}_{t\in[0,T)}$ be  a bounded smooth $\Phi$-curvature flow with forcing $f_i$ and anisotropic contact angle $\beta_i,$ $i=1,2.$
Then
\begin{equation}\label{strong_max_princp}
\beta_1\ge\beta_1,\quad f_1\ge f_2,\quad E_1(0)\prec E_2(0)\qquad\Longrightarrow\qquad E_1(t)\prec E_2(t),\quad t\in[0,T).
\end{equation}
In other words, $\cl{\pOmega E_1(t)}\cap \cl{\pOmega E_2(t)} = \emptyset$ for all $t\in [0,T)$ if so at $t=0.$
\end{theorem}

Further, we refer to assertion \eqref{strong_max_princp} as the \emph{strong comparison principle}. 

\begin{proof}
In view of Theorem \ref{teo:stability_mcf},  decreasing $f_i$ and $\beta_i$ a bit, it is enough to prove \eqref{strong_max_princp} when the inequalities between $\beta_i$ and $f_i$ are strict. For $t\in[0,T]$ let 
$$
a_\Phi^i(t,x):=\sd_{E_i(t)}^\Phi(x),\quad 
a^i(t,x):=\sd_{E_i(t)}(x),
$$
and 
$$
d_\Phi(t):=\min\{x\in\cl{\Omega}:\,\, a_\Phi^1(t,x) - a_\Phi^2(t,x)\},\quad 
d(t):=\min\{x\in\cl{\Omega}:\,\, a^1(t,x) - a^2(t,x)\}.
$$
Since $E(0)\prec F(0),$ by \eqref{compare_trunc_sdist}  $d(0),d_\Phi(0) > 0.$ By contradiction, assume that there exists $t_0\in(0,T)$ such that $d(t),d_\Phi(t) > 0$ in $(0,t_0)$ and $d(t_0)=d_\Phi(t_0) = 0.$ Thus there exists $x_0\in \cl{\pOmega E_1(t_0)} \cap \cl{\pOmega E_2(t_0)}$ and $d_\Phi(t_0) = a_\Phi^1(t_0,x_0)-a_\Phi^2(t_0,x_0)=0.$ 

First assume that $x_0\in \p \Omega$ and let 
$\eta:=\nabla\Phi(\nu_{E_1(t_0)}(x_0))-\nabla\Phi(\nu_{E_2(t_0)}(x_0)).$ Since $-\be_n$ is the outer unit normal to $\Omega,$ by the anisotropic contact angle condition  
$
\eta\cdot (-\be_n)= \beta_1 - \beta_2 > 0.
$
Thus, applying Proposition \ref{prop:regular_distance} (d) with $-\eta$ and recalling the definitions of $x_0$ and $d_\Phi(t_0)$ we find 
$$
0\le a_\Phi^1(t_0,x_0 - s\eta) - a_\Phi^1(t_0,x_0) -a_\Phi^2(t_0,x_0 - s\eta) + a_\Phi^2(t_0,x_0) = -s\Big(\Big[\nu_{E_1(t_0)}^{\Phi^o} - \nu_{E_2(t_0)}^{\Phi^o}\Big] \cdot \eta + o(1)\Big)
$$
as $s\to 0^+,$ where $\eta^{\Phi^o}$ is defined in \eqref{eta_Phio}.  Since $\nabla \Phi(\cdot)$ is strictly maximal monotone (as a subdifferential of convex functions) and positively $0$-homogeneous, 
$$
0\le -\Big[\nu_{E_1(t_0)}^{\Phi^o} - \nu_{E_2(t_0)}^{\Phi^o}\Big] \cdot \eta=  -\Big[\nu_{E_1(t_0)}^{\Phi^o} - \nu_{E_2(t_0)}^{\Phi^o}\Big]\cdot \Big[\nabla\Phi(\nu_{E_1(t_0)}^{\Phi^o})-\nabla\Phi(\nu_{E_2(t_0)}^{\Phi^o})\Big] <0,
$$
a contradiction. Thus, $x_0\in\Omega.$ By the time-smoothness of the flows $E_i(\cdot)$ there exists $\delta>0$ such that for any $t\in [t_0-\delta,t_0]$ minimum points of $f(t,\cdot) - g(t,\cdot)$ lies in $\Omega$ (basically the minimizers belong to a union of half-lines in $\Omega$ starting from $\p\Omega$ and crossing both $\pOmega E_1(t_0)$ and $\pOmega E_2(t_0)$ orthogonally). Therefore, using a Hamilton-type trick (see e.g. \cite[Chapter 2]{Mantegazza:2011}), we can show  
$$
d'(t) = \frac{\p}{\p t}\,\sd_{E_1(t)}(y_t) -\frac{\p}{\p t}\,\sd_{E_2(t)}(y_t),\quad t\in[t_0-\delta,t_0],
$$
where $y_t\in\pOmega E_2(t)$ is any point satisfying $d(t) = \sd_{E_1(t)}(y_t) - \sd_{E_2(t)}(y_t) .$
Let $z_t\in \pOmega E_1(t)$ and $u_t\in \pOmega E_2(t)$ be such that $d(t) = \d_{E_1(t)}(y_t) = |y_t-z_t|$ and $\d_{E_2(t)}(y_t) = |y_t-u_t|.$ By the minimality of $y_t$, $\nu_{E_1(t)}(z_t) = \nu_{E_2(t)}(u_t) = :\nu_0$ and $y_t,z_t,u_t$ lie on the same straight line parallel to $\nu_{E_2(t)}(u_t).$ Now applying \eqref{repos_sdist_curva} we find
$$
d'(t) = \kappa_{E_1(t)}^\Phi(z_t) - \kappa_{E_2(t)}^\Phi(u_t) + f_1(t,z_t) - f_2(t, u_t).
$$
By the minimality of $y_t\in\pOmega E_2(t)$ and smoothness and the ellipticity of $\Phi,$ translating $E_1(t)$ along $\nu_{E_2(t)}(u_t)$ until we reach to $u_t\in \p E_2(t)$ we deduce that $\tilde E_1(t)\subset E_2(t)$ and $\p \tilde E_1(t)$ is tangent to $\p E_2(t)$ at $u_t,$ where $\tilde E_1(t)$ is the translated $E_1(t).$ Then $\kappa_{E_1(t)}^\Phi(z_t)= \kappa_{\tilde E_1(t)}^\Phi(u_t)\ge \kappa_{E_2(t)}^\Phi(u_t)$ and therefore, by the  $C^{\alpha/2,\alpha}$-regularity of $f_2,$  
$$
d'(t) \ge f_1(t,z_t) - f_2(t,u_t) = f_1(t,z_t) - f_2(t,z_t + d(t)\nu_0) \ge f_1(t,z_t) - f_2(t,z_t) - C_{f_2} d(t)^\alpha,
$$
where $C_{f_2}$ is the H\"older constant of $f_2.$ Since $\{E_i(t)\}$ is bounded uniformly in $t\in [t_0-\delta,t_0]$ and by assumption $f_1>f_2,$ there exists $\gamma_0>0$ independent of $t$ such that $f_1(t,z_t) - f_2(t,z_t)\ge \gamma_0.$ Thus, recalling the continuity of $d(\cdot)$ and assumption $d(t_0)=0$ possibly decreasing $\delta$ a bit, we get $d'(t)>\gamma_0/2$ for any $t\in [t_0-\delta,t_0].$ Therefore, $d$ is strictly increasing in this interval so that $0=d(t_0)>d(t_0-\delta)>0,$ a contradiction.

These contradictions show that $\cl{\pOmega  E(t)}\cap \cl{\pOmega F(t)} =\emptyset$ for any $t\in [0,T).$ Hence, $E(t)\prec F(t).$
\end{proof}

\subsubsection{Evolution of tubular neighborhoods}

Recall that a crucial part in the proof of the consistency in \cite[Theorem 7.4]{ATW:1993} is the evolution of tubular neighborhoods \cite[Corollary 7.2]{ATW:1993} which is given by the level sets of signed distance functions. Unfortunately, in our setting due to the contact angle condition we cannot use signed distances. Therefore, as in \cite{Kholmatov:2023} we construct a sort of tubular neighborhoods, which possess similar properties of the true tubular neighborhoods in case of without boundary, important in the proof of the consistency.

To this aim, in the following lemma we define a ``foliation'' of a tubular neighborhood of the boundary of an admissible set, consisting of boundaries of admissible families with a prescribed anisotropic contact angle.

\begin{lemma}[\textbf{Foliations}]
Let $E_0$ be an admissible set with anisotropic contact angle $\beta$. Then there exist positive numbers $\rho\in(0,1)$ and $\sigma\in (0,\eta),$ depending  only\footnote{We ignore the dependence on $\alpha$ and $\eta.$} on $\|II_{E_0}\|_\infty$ and $h_{E_0}$ (see \eqref{minheights}), and admissible families $\{G_0^\pm[r,s]\}_{(r,s)\in[0,\rho]\times[0,\sigma]}$ such that $G_0^\pm[0,0]=E_0$ and for all $(r,s)\in [0,\rho]\times [0,\sigma]$:

\begin{itemize}[itemsep=3pt]
\item[\rm(a)] $\dist(\pOmega G_0^\pm[r,s], \pOmega E_0)\ge r+s$ and
\begin{align*}
& G_0^-[r,s] \subset E_0 \subset G_0^+[r,s]\\
& \dist(\pOmega G_0^\pm[r,s],\pOmega G_0^\pm[0,s])=r,\\
&\dist(\pOmega G_0^\pm[0,s],\pOmega E_0)=s;
\end{align*}

\item[\rm(b)] $G_0^\pm[r,s]$ is admissible with anisotropic contact angle $\beta\mp s.$ 

\end{itemize}
\end{lemma}

Notice that this lemma is a generalization of \cite[Lemma 2.4]{Kholmatov:2024} to the anisotropic setting can be done along the same lines.

\begin{corollary}\label{cor:time_foliations}
Let $\{E[t]\}_{t\in[0,T)}$ be an admissible family contact angle $\beta.$ Then for any $T'\in(0,T)$ there exist $\rho\in(0,1)$ and $\sigma\in (0,\eta)$ depending only $\sup_{t\in[0,T']} \|II_{E[t]}\|_\infty$ and $\inf_{t\in[0,T']} h_{E[t]},$ and admissible families $\{G_0^\pm[r,s,a]\}_{(r,s,a)\in[0,\rho]\times[0,\sigma]\times[0,T']}$ such that $G_0^\pm[0,0,a]=E[a]$ and for all $(r,s,a)\in [0,\rho]\times [0,\sigma]\times [0,T']$:

\begin{itemize}[itemsep=3pt]
\item[\rm(a)]  $\dist(\pOmega G_0^\pm[r,s,a], \pOmega E[a])\ge r+s$ and
\begin{align*}
&G_0^-[r,s,a] \subset E[a] \subset G_0^+[r,s,a],\\
&\dist(\pOmega G_0^\pm[r,s,a],\pOmega G_0^\pm[0,s,a])=r,\\
&\dist(\pOmega G_0^\pm[0,s,a],\pOmega E[a])=s;
\end{align*}

\item[\rm(b)] $G_0^\pm[r,s,a]$ is admissible with anisotropic contact angle $\beta\mp s.$ 
\end{itemize}
\end{corollary}

By the definition of the admissibility, $G_0^\pm[r,s,a]$ is close to $E[a]$ in the sence of Definition \ref{def:def:distance_drops}. Therefore, applying Theorem \ref{teo:stability_mcf} we deduce

\begin{theorem}\label{teo:flow_tubular_nbhd}
Let a family $\{E[t]\}_{t\in [0,T^\dag)}$ of admissible sets be a $\Phi$-curvature flow with forcing $f$ and anisotropic contact angle $\beta,$ and let $T\in (0,T^\dag).$ Let $\rho\in(0,1)$ and $\sigma\in(0,\eta),$ and for $a\in [0,T),$ the families $\{G_0^\pm[r,s,a]\}_{(r,s,a)\in[0,\rho]\times [0,\sigma]\times[0,T']}$ be given by Corollary \ref{cor:time_foliations}. 
Then (possibly decreasing $\rho$ and $\sigma$ slightly, depending only on $\{E(t)\}$) there exist unique admissible families $\{G^\pm[r,s,a,t]\}_{(r,s,a)\in[0,\rho]\times [0,\sigma]\times[0,T'],t\in[a,T]}$ such that
\begin{itemize}
\item $G^\pm[r,s,a,a]=G_0^\pm[r,s,a],$

\item $G^\pm[r,s,a,t]$ is admissible with anisotropic contact angle $\beta\mp s,$ 

\item 
\begin{equation}\label{forced_curva}
v_{G^\pm[r,s,a,t]}(x) = - \kappa_{G^\pm[r,s,a,t]}(x) -f(t,x) \pm s\quad\text{for $t\in (a,T)$ and $x\in \pOmega G^\pm[r,s,a,t].$}
\end{equation}
\end{itemize}
Furthermore, 
\begin{itemize}
\item[\rm(a)] $G^\pm[0,0,a,t]=E[t]$ for all $t\in[a,T];$

\item[\rm(b)] there exists an increasing continuous function $g:[0,+\infty)\to[0,+\infty)$ with $g(0)=0$ such that 
$$
\max_{x\in \pOmega G^\pm[0,s,a,t]}\,\,\dist(x,\pOmega G^\pm[0,0,a,t]) \le g(s)
$$
for all $s\in[0,\sigma],$ $a\in[0,T]$ and $t\in[0,T];$

\item[\rm(c)] there exists $t^*\in (0,\rho/64)$ (independent of $r,s$ and $a$) such that 
\begin{equation}\label{military_conflict}
G_0^+[\rho/2,s,a] \subset G^+[\rho,s,a,a+t']
\quad 
\text{and}\quad 
G_0^-[\rho/2,s,a] \supset G^-[\rho,s,a,a+t']
\end{equation}
for all $t'\in [0,t^*]$ with $a+t'\le T.$
\end{itemize} 
\end{theorem}

Notice that the assertions (a)-(c) follow from the continuous dependence of $G^\pm$ on $[r,s,a,t]$.  In view of \eqref{repos_sdist_curva} we can represent  \eqref{forced_curva} as
$$
\tfrac{\p}{\p t}\,\sd_{G^\pm[r,s,a,t]}(x) = \kappa_{G^\pm[r,s,a,t]}(x) + f(t,x) \mp s\qquad \text{for $t\in(a,T)$ and $x\in \pOmega  G^\pm[r,s,a,t].$}
$$

\begin{proposition}\label{prop:time_regular_sdist}
For any $s\in(0,\sigma]$ there exists $\tau_0(s)>0$ such that for any $r\in[0,\rho],$ $a\in[0,T),$ $\tau\in(0,\tau_0)$ and $t\in [a+\tau,T]$ 
$$
\tfrac{\sd_{G^+[r,s,a,t - \tau]}(x)}{\tau} + \kappa_{G^+[r,s,a,t]}(x) + 
\frac{1}{\tau}\int_{k\tau}^{(k+1)\tau}
f(s,x)ds > \tfrac{s}{2},\quad x\in \pOmega G^+[r,s,a,t],
$$
and
$$
\tfrac{\sd_{G^-[r,s,a,t - \tau]}(x)}{\tau} + \kappa_{G^-[r,s,a,t]}(x) + \frac{1}{\tau}\int_{k\tau}^{(k+1)\tau}
f(s,x)ds < - \tfrac{s}{2},\quad x\in \pOmega G^+[r,s,a,t],
$$
where $k:=\intpart{t/\tau}.$
\end{proposition}

This result is proven along the same lines of \cite[Proposition 2.7]{Kholmatov:2023}.  Therefore, we omit it.

\section{Consistency of GMM with smooth $\Phi$-curvature flow}\label{sec:consistence}

In this section we prove Theorem \ref{teo:gmm_vs_smooth}.

Let $\{E(t)\}_{t\in [0,T^\dag)}$ be a smooth $\Phi$-curvature flow starting from $E_0,$ with a bounded forcing $f$ and anisotropic contact angle $\beta.$ Given $T\in (0,T^\dag),$ let $\rho,$ $\sigma,$ $\{G_0^\pm[r,s,a]\},$ $\{G^\pm [r,s,a,t]\}$ and $t^*>0$ be as in Theorem \ref{teo:flow_tubular_nbhd}. 
Let also $F(\cdot)\in GMM(\sF_{\beta,f},E_0),$
$\tau_j\searrow 0$ and $\{F(\tau_j,k)\}$ be such that 
\begin{equation}\label{flats_converge}
\lim\limits_{j\to+\infty} |F(\tau_j,\intpart{t/\tau_j})\Delta F(t)| =0\quad\text{for all $t\ge0.$}
\end{equation}
We show 
\begin{equation}\label{shortly_equals}
E(t)=F(t)\quad\text{for any $0 < t < T.$}
\end{equation}

We start with an ancillary technical lemma.
For $s\in(0,\sigma]$ let $\tau_0(s)>0$ be given by Proposition \ref{prop:time_regular_sdist} and for $\beta_0:=\frac{\|\beta\|_\infty + \Phi(\be_n)}{2} \in (\|\beta\|_\infty,\Phi(\be_n)),$ let $\vartheta_0$ be given by Theorem \ref{teo:compare_with_ball}. We may assume that $\tau_j<\vartheta_0\rho^2/64^2$ for all $j.$ 

\begin{lemma}\label{lem:chaklpakes}
Assume that $t_0\in[0,T)$ and $k_0\in\N_0$ are such that 
\begin{equation}\label{initially_Good}
G_0^-[0,s,t_0] \subset F(\tau_j,k_0) \subset G_0^+[0,s,t_0].
\end{equation}
Then there exists $\bar t\in(0,t^*]$ depending only on $t^*$ and $\rho$ such that 
\begin{equation*}
G^-[0,s,t_0, t_0+k\tau_j] \subset F(\tau_j,k_0+k) \subset G^+[0,s,t_0,t_0+k\tau_j]
\end{equation*}
for all $s\in(0,\sigma],$ $j\ge1$ with $\tau_j\in(0,\tau_0(s))$ and $k=0,1,\ldots, \intpart{\bar t/\tau_j}$ with $t_0+k\tau_j<T.$ 
Moreover, let $t_0+\bar t < T,$ the increasing continuous function $g$ be given by Theorem \ref{teo:flow_tubular_nbhd} (b) and $\bar\sigma\in(0,\sigma/2)$ be such that $4g(2\bar\sigma)<\sigma.$ Then for any $s\in(0,\bar\sigma)$ there exists $\bar j(s)>1$ such that 
\begin{equation}\label{tiktok_videos}
G_0^-[0,4g(2s), t_0+\bar t] \subset F(\tau_j,k_0+\bar k_j) \subset G_0^+[0,4g(2s), t_0+\bar t]
\end{equation}
whenever $j>\bar j(s),$ where $\bar k_j:=\intpart{\bar t/\tau_j}.$
\end{lemma}

\begin{proof}
The proof runs along the similar lines of \cite[Lemma 3.1]{Kholmatov:2023}.
By Corollary \ref{cor:time_foliations} (a) 
$$
\dist(\pOmega G_0^\pm[\rho/4,s,t_0],\pOmega G_0^\pm[0,s,t_0]) = \rho/4,
$$ 
and therefore,  by \eqref{initially_Good} 
\begin{equation}\label{init_Better}
G_0^-[\rho/4,s,t_0]\prec G_0^-[0,s,t_0] \subset F(\tau_j,k_0) \subset G_0^+[0,s,t_0]\prec G_0^+[\rho/4,s,t_0]
\end{equation}
and by \eqref{init_Better} $B_{\rho/4}(x)\subset F(\tau_j,k_0) $ if $ x\in G_0^-[\rho/4,s,t_0]$ and 
$B_{\rho/4} (x)\cap F(\tau_j,k_0) =\emptyset $ if $x\in \Omega\setminus G_0^+[\rho/4,s,t_0].$ Therefore, using Theorem \ref{teo:compare_with_ball} (with $R_0=\rho/4$ and $\beta_0:= \frac{\Phi(\be_n)+\|\beta\|_\infty}{2}$) and again \eqref{initially_Good} we obtain
\begin{equation}\label{bolls0987}
\begin{cases}
B_{\frac{\beta_0\rho}{64\Phi(\be_n)}} (x)\subset F(\tau_j,k_0+k) &  x\in G_0^-[\rho/4,s,t_0],\\
B_{\frac{\beta_0\rho}{64\Phi(\be_n)}} (x)\cap F(\tau_j,k_0+k) =\emptyset & x\in \Omega\setminus G_0^+[\rho/4,s,t_0],
\end{cases}
\quad k=0,1,\ldots,\intpart{t^{**}/\tau_j},
\end{equation}
where 
$$
t^{**}:= \tfrac{\vartheta_0\rho^2}{16}.
$$
By \eqref{bolls0987} and Corollary \ref{cor:time_foliations} (b)
\begin{equation}\label{qoshona}
G_0^-\Big[\tfrac{\rho}{2}, s,t_0\Big]\subset 
G_0^-\Big[\tfrac{\rho}{4}- \tfrac{\beta_0\rho}{64}, s,t_0\Big] \subset F(\tau_j,k_0+k) \subset G_0^+\Big[\tfrac{\rho}{4}- \tfrac{\beta_0\rho}{64}, s,t_0\Big] \subset G_0^+\Big[\tfrac{\rho}{4}, s,t_0\Big]
\end{equation}
for all $0\le k\le \intpart{t^{**}/\tau_j}.$ 
Set 
$$
\bar t:=\min\Big\{t^*,t^{**}\Big\},
$$ 
where $t^*$ is given by Theorem \ref{teo:flow_tubular_nbhd} (c). Then by \eqref{military_conflict} and \eqref{qoshona} 
\begin{equation}\label{qoshona_new}
G_0^-[\rho, s,t_0+k\tau_j] \subset F(\tau_j,k_0+k) \subset G_0^+[\rho, s,t_0+ k\tau_j],\quad k=0,1,\ldots, \intpart{\bar t/\tau_j},
\end{equation}
with $t_0+k\tau_j<T.$ 
We claim for such $k$  and $j\ge1$ with $\tau_j\in(0,\tau_0(s))$
\begin{equation*}
G^-[0,s,t_0,t_0+k\tau_j] \subset F(\tau_j,k_0+k) \subset G^+[0,s,t_0, t_0+k\tau_j].
\end{equation*}
Indeed,  let 
$$
\bar r:=\inf\Big\{r\in [0,\rho]:\,\,\, F(\tau_j,k_0+k) \subset G^+[r,s,t_0,t_0+k\tau_j]\quad k=0,1,\ldots,\intpart{\bar t/\tau_j},\,\,t_0+k\tau_j<T\Big\}.
$$
By \eqref{qoshona_new} the infimum is taken over a nonempty set. By contradiction, assume that $\bar r>0.$ In view of the continuity of $G^+[r,s,t_0,t_0+k\tau_j]$ at $r=\bar r,$ there exists the smallest integer $k\le \intpart{\bar t/\tau_j}$ (clearly, $k>0$ by \eqref{init_Better}) such that 
\begin{equation}\label{chandui_jahon098}
\cl{\pOmega F(\tau_j,k_0+k)} \cap \cl{\pOmega G^+[\bar r,s,t_0,t_0+k\tau_j]} \ne\emptyset.
\end{equation}
By the minimality of $k\ge1$ 
$$
F(\tau_j,k_0+k-1)\subset G^+[\bar r,s,t_0,t_0+(k-1)\tau_j],\qquad 
F(\tau_j,k_0+k)\subset G^+[\bar r,s,t_0,t_0 + k\tau_j].
$$
Moreover, by construction $G^+[\bar r,s,t_0,t_0+k\tau_j]$ satisfies the contact angle condition with $\beta-s$ and by Proposition \ref{prop:time_regular_sdist} applied with $\tau=\tau_j\in(0,\tau_0(s))$
$$
\tfrac{\sd_{G^+[\bar r,s,t_0, t_0+(k-1)\tau_j]}(x)}{\tau_j} + \kappa_{G^+[\bar r,s,t_0,t_0 + k\tau_j]}(x) + \frac{1}{\tau}\int_{k\tau}^{(k+1)\tau}
f(s,x)ds > \tfrac{s}{2},\quad x\in \pOmega G^+[\bar r,s,t_0,t_0 + k\tau_j].
$$
However, in view of Proposition \ref{prop:atw_nonpol} (a), these properties imply $F(\tau_j,k_0+k)\prec G^+[\bar r,s,t_0,t_0+k\tau_j],$ which contradicts to \eqref{chandui_jahon098}. Thus, $\bar r=0.$  Analogous contradiction argument based on Proposition \ref{prop:atw_nonpol} (b) shows $G^-[0,s,t_0,t_0+k\tau_j] \subset F(\tau_j,k_0+k)$ for all $0\le k\le \intpart{\bar t/\tau_j}.$

Finally, let us prove \eqref{tiktok_videos}. Recall that by construction $G_0^-[0,2s,t_0]\prec G_0^-[0,s,t_0]$ and $G_0^+[0,s,t_0]\prec G_0^+[0,2s,t_0],$ therefore, by the strong comparison principle (Theorem \ref{teo:comparison_mcf}) $G^-[0,2s,t_0,t]\prec G^-[0,s,t_0,t]$ and $G^+[0,s,t_0,t]\prec G^+[0,2s,t_0,t]$ for all $t\in[t_0,T].$ Now the continuity of $G^\pm[0,s,t_0,t]$ on its parameters we could find $\bar j=\bar j(s)>1$ such that for all $j>\bar j$
\begin{multline}\label{mahszrte}
G^-[0,2s,t_0,t_0+\bar t] \prec G^-[0,s,t_0,t_0+\bar k_j\tau_j] \\
\subset F(\tau_j,\bar k_j) \subset G^+[0,s,t_0,t_0 +\bar k_j\tau_j]\prec G^+[0,2s,t_0,t_0+\bar t].
\end{multline}
By the definition of $g,$
\begin{equation}\label{copsos}
\max\limits_{x\in \pOmega G^\pm[0,2s,t_0,t_0+ \bar t]}\,\, \dist(x, \pOmega E(t_0+ \bar t)) \le g(2s)
\end{equation}
and therefore, by construction in Corollary \ref{cor:time_foliations} (a) 
$$
\dist(\pOmega G_0^\pm[0,4g(2s),t_0+\bar t],\pOmega E(t_0+ \bar t)) = 4g(2s)>0.
$$
Combining this with \eqref{copsos} and the construction of $G_0^\pm$ we deduce 
$$
G_0^-[0,4g(2s),t_0+\bar t] \prec G^-[0,2s,t_0,t_0+\bar t] 
\quad\text{and}\quad 
G^+[0,2s,t_0,t_0+\bar t]
\prec 
G_0^+[0,4g(2s),t_0+\bar t].
$$
These inclusions together with \eqref{mahszrte} imply \eqref{tiktok_videos}.
\end{proof}

Now we are ready to prove the equality \eqref{shortly_equals}. Let $\bar t$ be given by Lemma \ref{lem:chaklpakes},
$$
N:=\intpart{T/\bar t} + 1
$$
and let $\sigma_0\in(0,\sigma/16)$ be such that the numbers
$$
\sigma_l=4g(2\sigma_{l-1}),\quad l=1,\ldots,N,
$$
satisfy $\sigma_l\in (0,\sigma/16).$ By the monotonicity and continuity of $g$ together with $g(0)=0,$ such choice of $\sigma_0$ is possible.

Fix any $s\in(0,\sigma_0)$ and let 
$$
a_0(s):=s,\quad a_l(s):=4g(2a_{l-1}(s)),\quad l=1,\ldots,N.
$$
Note that $a_l(s)\in (0,\sigma_l).$ In particular, the numbers $\bar j_l^s:=\bar j(a_l(s)),$ given by the last assertion of Lemma \ref{lem:chaklpakes}, are well-defined. Let also 
$$
\tilde j_l^s:=\max\{j\ge1:\,\, \tau_j\notin (0,\tau_0(a_l(s)))\}
$$
and
$$
\bar j_s:= 1 + \max_{l=0,\ldots,N}\,\max \{\bar j_l^s,\tilde j_l^s\}.
$$

By Corollary \ref{cor:time_foliations} (a)
$$
G_0^-[0,s,0]\subset E(0)=E_0=F(\tau_j,0) \subset G_0^+[0,s,0]
$$
for all $j>\bar j_s.$
Therefore, by Lemma \ref{lem:chaklpakes} applied with $k_0=0$ and $t_0=0$ we find 
$$
G^-[0,s,0,k\tau_j]\subset F(\tau_j,k) \subset G^+[0,s,0,k\tau_j],\quad k=0,1,\ldots,\bar k_j,
$$
where $\bar k_j:=\intpart{\bar t/\tau_j}.$ Moreover, since $s\in (0,\sigma_0,)$ by the last assertion of Lemma \ref{lem:chaklpakes} 
$$
G_0^-[0,a_1(s),\bar t]\subset F(\tau_j,\bar k_j) \subset G_0^+[0,a_1(s),\bar t] 
$$
for all $j\ge \bar j_s.$ Hence, we can reapply Lemma \ref{lem:chaklpakes} with $s:=a_1(s),$ $t_0=\bar t$ and $k_0=\bar k_j,$ to find 
$$
G^-[0,a_1(s),0,\bar t + k\tau_j]\subset F(\tau_j,\bar k_j+k) \subset G^+[0,a_1(s),0,\bar t+ k\tau_j],\quad k=0,1,\ldots,\bar k_j.
$$
In particular, since $j>\bar j_s> \bar j(a_1(s)),$ again by the last assertion of  Lemma \ref{lem:chaklpakes}  we deduce 
$$
G_0^-[0,a_2(s),2\bar t]\subset F(\tau_j,2\bar k_j) \subset G_0^+[0,a_2(s),2\bar t]. 
$$
Repeating this argument at most $N$ times, for all $j\ge \bar j_s$ we find 
\begin{equation}\label{nadomatlar671}
G^-[0,a_l(s),0,l\bar t + k\tau_j]\subset F(\tau_j,l\bar k_j+k) \subset G^+[0,a_l(s),0,l\bar t + k\tau_j],\quad k=0,1,\ldots,\bar k_j
\end{equation}
whenever $l=0,\ldots,N$ and $l\bar t + k\tau_j\le T.$ 

Now take any $t\in(0,T),$ and let $l:=\intpart{t/\bar t}$  and $k=\intpart{t/\tau_j} - l\bar k_j$ so that $l\bar k_j + k = \intpart{t/\tau_j}.$ By means of $l$ and $k,$ as well as the definition of $\bar k_j$ we represent \eqref{nadomatlar671} as 
\begin{align}\label{headlight09}
G^-\Big[0,a_l(s),0,l\bar t + \tau_j\intpart{\tfrac{t}{\tau_j}} - l\tau_j\intpart{\tfrac{\bar t}{\tau_j}}\Big]
\subset 
F\Big(\tau_j, \intpart{\tfrac{t}{\tau_j}}\Big) 
\subset 
G^+\Big[0,a_l(s),0,l\bar t + \tau_j\intpart{\tfrac{t}{\tau_j}} - l\tau_j\intpart{\tfrac{\bar t}{\tau_j}}\Big]
\end{align}
for all $j>\bar j_s.$ Since 
$$
\lim\limits_{j\to+\infty} \Big(l\bar t + \tau_j\intpart{\tfrac{t}{\tau_j}} - l\tau_j\intpart{\tfrac{\bar t}{\tau_j}}\Big) = t,
$$
by the continuous dependence of $G^\pm$ on its parameters, as well as the convergence \eqref{flats_converge} of the flat flows, letting $j\to+\infty$ in \eqref{headlight09} we obtain
\begin{equation}\label{endiyonimga}
G^-[0,a_l(s),0,t]\subset F(t) \subset G^+[0,a_l(s),0,t],
\end{equation}
where due to the $L^1$-convergence the inclusions in \eqref{flats_converge} are up to some negligible sets. Now we let $s\to0^+$ and recalling that $a_l(s)\to0$ (by the continuity of $g$ and assumption $g(0)=0$), from \eqref{endiyonimga} we deduce 
\begin{equation*}
G^-[0,0,0,t]\subset F(t) \subset G^+[0,0,0,t].
\end{equation*}
Then by Theorem \ref{teo:flow_tubular_nbhd} (a) 
$$
F(t)=G^\pm[0,0,0,t] = E(t).
$$

\appendix 
\section{Some useful results}

The following lemma extends analogous results in the Euclidean case \cite[Sections 2 and 3]{BKh:2018}.

\begin{lemma}[\textbf{A priori estimates for capillary functional}]

Let $\beta\in L^\infty(\p\Omega).$ Then:
\begin{itemize}
\item[\rm(a)] for any $E\in BV(\Omega;\{0,1\})$
\begin{equation}\label{qanday_kungil_uzsam}
\tfrac{\Phi(\be_n) + \inf\,\beta}{2\Phi(\be_n)}\, P_\Phi(E) \le \sC_\beta(E) \le \max\Big\{\tfrac{\sup\,\beta}{\Phi(\be_n)} ,1\Big\} P_\Phi(E);
\end{equation} 

\item[\rm(b)] $\sC_\beta$ is $L_\loc^1(\Omega)$-lower semicontinuous if and only if $\|\beta\|_\infty \le \Phi(\be_n).$ 

\end{itemize}
\end{lemma}

\begin{proof}
(a) The upper bound is clear. To prove the lower bound, let $\beta_o:=\inf\,\beta.$
By the anisotropic minimality of the halfspaces (see e.g. \cite[Example 2.4]{BKhN:2017})
\begin{align}\label{min_halfplanes}
P_\Psi(E) \ge \Psi(\be_n) \int_{\p\Omega}\chi_E\,d\cH^{n-1}
\end{align}
for any anisotropy $\Psi$ in $\R^n.$ Thus, if $\beta_o\ge0,$ then  by \eqref{min_halfplanes}
\begin{multline*}
\sC_\beta(E) \ge \gamma\,P_\Phi(E,\Omega) + (1-\gamma) \,P_\Phi(E,\Omega) + \beta_o \int_{\p\Omega} \chi_E\,d\cH^{n-1} \\
\ge \gamma\, P_\Phi(E,\Omega) + \Big(1-\gamma + \tfrac{\beta_o}{\Phi(\be_n)}\Big)\,\Phi(\be_n)\int_{\p\Omega}\chi_E\,d\cH^{n-1}
\end{multline*}
and hence, choosing $\gamma = \frac{\Phi(\be_n) + \beta_o}{2\Phi(\be_n)}$
we deduce \eqref{qanday_kungil_uzsam}.

Assume that $\beta_o<0.$ Then  $\Psi:=\frac{\Phi(\be_n) - \beta_o}{2\Phi(\be_n)}\,\Phi$ is an anisotropy. By \eqref{min_halfplanes}
$$
\int_{\Omega\cap \p^*E} \Psi(\nu_E)\,d\cH^{n-1} \ge \Psi(\be_n)\int_{\p\Omega}\chi_E\,d\cH^{n-1} = \tfrac{\Phi(\be_n) - \beta_o}{2}\,\int_{\p\Omega}\chi_E\,d\cH^{n-1}.
$$
Thus,
\begin{align*}
\sC_\beta(E) = & \int_{\Omega\cap\p^*E} \tfrac{\Phi(\be_n) + \beta_o}{2\Phi(\be_n)}\,\Phi(\nu_E)\,d\cH^{n-1} +  \int_{\Omega\cap\p^*E} \Psi(\nu_E)\,d\cH^{n-1} + \int_{\p\Omega}\beta\chi_E\,d\cH^{n-1}\\[2mm]
\ge & \tfrac{\Phi(\be_n) + \beta_o}{2\Phi(\be_n)}\,P_\Phi(E,\Omega) + \int_{\p\Omega} \tfrac{\Phi(\be_n) - \beta_o+2\beta}{2}\chi_E\,d\cH^{n-1}\\[2mm]
\ge & \tfrac{\Phi(\be_n) + \beta_o}{2\Phi(\be_n)}\,P_\Phi(E).
\end{align*}

(b) Repeat arguments of \cite[Lemma 3.5]{BKh:2018}.
\end{proof}

The following proposition provides a characterization of elliptic $C^2$-norms. 

\begin{proposition}\label{prop:elliptic_anis_propo}
For any $C^{k+\alpha}$-anisotropy $\Phi$ with $k\ge2$ and $\alpha\in[0,1],$ the following assertions are equivalent:
\begin{itemize}
\item[\rm(a)] $\Phi$ is elliptic;

\item[\rm(b)] there exists $\gamma>0$ such that 
$$
\nabla^2\Phi(x)y^T \cdot y^T \ge \frac{\gamma}{|x|}\quad 
\text{for any $x\in\R^n\setminus\{0\},$ $y\in\S^{n-1}$ with $x\cdot y = 0;$}
$$

\item[\rm(c)] there exists $\gamma>0$ such that 
$$
\nabla^2\Phi(x)y^T \cdot y^T \ge \frac{\gamma}{|x|}\quad 
\text{for any $x\in\R^n\setminus\{0\},$ $y\in\S^{n-1}$ with $\nabla \Phi(x)\cdot y = 0;$} 
$$

\item[\rm(d)] there exists $\gamma>0$ such that
$$
\nabla^2\Phi(x)y^T\cdot y^T \ge \tfrac{\gamma}{|x|} \Big|y - \Big(y\cdot \tfrac{x}{|x|}\Big)\tfrac{x}{|x|}\Big|^2\quad\text{for any $x\in\R^n\setminus\{0\},$ $y\in\R^n;$}
$$

\item[\rm(e)] for any segment $[x,y],$ lying on the line not passing through the origin, the second derivative of the function $t\mapsto h(t):=\Phi(x + t(y-x))$ is strictly positive in $[0,1]$;

\item[\rm(f)] $\Phi^o$ is $C^{k+\alpha}$ and elliptic;

\item[\rm(g)] the principal curvatures of the boundary of $W^\Phi$ is strictly positive;

\item[\rm(h)] there exists $r\in(0,1)$ such that for any $z\in\p  W^\Phi$ there exist $x_z,y_z\in\R^n$ such that 
$$
B_r(x_z)\subset W^\Phi \subset \cl{B_{1/r}(y_z)}
\qquad\text{and}\qquad 
\p B_r(x_z)\cap \p W^\Phi =\p B_{1/r}(y_z)\cap \p  W^\Phi = \{z\}.
$$
\end{itemize}
\end{proposition}

\begin{proof}
Since $\Phi$ is $C^2$ and 
$$
\nabla^2\Phi(x)x^T = 0,\quad x\in \R^n\setminus\{0\},
$$
the ellipticity of $\Phi$ is equivalent to the strict positivity of
its Hessian $\nabla^2 \Phi(x)$ on $T_x:=\{y:\,x\cdot y=0\}.$ Thus, passing to local coordinates, one can show (a)$\Rightarrow$(g)$\Rightarrow$(h)$\Rightarrow$(g)$\Rightarrow$(a). 
Moreover, the assertions 
$$
\text{(a)$\Rightarrow$(b)$\Rightarrow$(c)$\Rightarrow$(b)$\Rightarrow$(d)$\Rightarrow$(e)$\Rightarrow$(b)$\Rightarrow$(a)}
$$
follow directly from the definition of ellipticity.

Finally, let us show (a)$\Rightarrow$(f). Since $\p W^\Phi$ does not contain segments and $$
\nabla \Phi(x)\cdot x = \Phi(x) \quad \text{and} \quad 
\Phi^o(\nabla\Phi(x)) = 1,\quad x\in\R^n,
$$
$\Phi^o$ is differentiable on $\R^n\setminus\{0\}.$ Hence, by convexity, $\Phi^o$ is $C^1.$   
The implications (g)$\Leftrightarrow$(h) follows from the fact that the second fundamental form of $\p W^\Phi$ is bounded from below and from above by that of ball. Similar arguments can be used in the proof of the implication (a)$\Rightarrow$(h) using the strict positive definiteness of $\nabla^2\Phi(x)$ on $T_x$ (in view of the convexity of $x\mapsto \Phi(x)-\gamma|x|$).

Finally, we prove (b)$\Rightarrow$(f). Since $\p W^\Phi$ has no segments, $\Phi^o$ is differentiable in $\R^n\setminus\{0\},$ and by convexity, $\nabla \Phi^o$ is continuous. Since the map $x\in\p W^\Phi\mapsto \nabla\Phi(x) \in \p W^{\Phi^o}$ is a homeomorphism. By (b) and (c),
$$
\nabla^2\Phi(x)y^T\cdot y^T \ge \frac{\gamma}{|x|}
$$
for any $x\in\p W^\Phi$ and $y\in\S^{n-1}$ with $x\cdot y=0.$ This implies $\nabla^2\Phi$ maps the tangent plane of $\p W^\Phi$ at $x$ to the tangent plane of $\p W^{\Phi^o}$ at $\nabla\Phi(x).$
Thus, by the inverse mapping theorem, the $\nabla \Phi$ is a $C^{k-1+\alpha}$-homeomorphism. In particular,  $\p W^{\Phi^o}$ is locally a $C^{k+\alpha}$-manifold, and hence, $\Phi^o$ is $C^{k+\alpha}.$ Finally, to prove the ellipticity it is enough to observe 
$$
\nabla^2\Phi^o(x)y^T\cdot y^T >0
$$
for any $x\in \p W^{\Phi^o}$ and $y\in \S^{n-1}$ with $y\cdot \nabla\Phi^o(x)=0,$ thus, assertion (c)  holds, and hence, $\Phi^o$ is also elliptic.
\end{proof}

\end{document}